\documentclass[11pt,a4paper]{article}
\usepackage{tikz}
\usepackage{subfigure}
\usepackage[english]{babel}
\usepackage{graphicx}
\usepackage{extarrows}

\usepackage{indentfirst, latexsym, bm, amssymb}

\begin{document}

\newtheorem{theorem}{Theorem}[section]
\newtheorem{corollary}[theorem]{Corollary}
\newtheorem{definition}[theorem]{Definition}
\newtheorem{conjecture}[theorem]{Conjecture}
\newtheorem{question}[theorem]{Question}
\newtheorem{lemma}[theorem]{Lemma}
\newtheorem{proposition}[theorem]{Proposition}
\newtheorem{example}[theorem]{Example}
\newenvironment{proof}{\noindent {\bf
Proof.}}{\rule{3mm}{3mm}\par\medskip}
\newcommand{\remark}{\medskip\par\noindent {\bf Remark.~~}}
\newcommand{\pp}{{\it p.}}
\newcommand{\de}{\em}

\newcommand{\JEC}{{\it Europ. J. Combinatorics},  }
\newcommand{\JCTB}{{\it J. Combin. Theory Ser. B.}, }
\newcommand{\JCT}{{\it J. Combin. Theory}, }
\newcommand{\JGT}{{\it J. Graph Theory}, }
\newcommand{\ComHung}{{\it Combinatorica}, }
\newcommand{\DM}{{\it Discrete Math.}, }
\newcommand{\ARS}{{\it Ars Combin.}, }
\newcommand{\SIAMDM}{{\it SIAM J. Discrete Math.}, }
\newcommand{\SIAMADM}{{\it SIAM J. Algebraic Discrete Methods}, }
\newcommand{\SIAMC}{{\it SIAM J. Comput.}, }
\newcommand{\ConAMS}{{\it Contemp. Math. AMS}, }
\newcommand{\TransAMS}{{\it Trans. Amer. Math. Soc.}, }
\newcommand{\AnDM}{{\it Ann. Discrete Math.}, }
\newcommand{\NBS}{{\it J. Res. Nat. Bur. Standards} {\rm B}, }
\newcommand{\ConNum}{{\it Congr. Numer.}, }
\newcommand{\CJM}{{\it Canad. J. Math.}, }
\newcommand{\JLMS}{{\it J. London Math. Soc.}, }
\newcommand{\PLMS}{{\it Proc. London Math. Soc.}, }
\newcommand{\PAMS}{{\it Proc. Amer. Math. Soc.}, }
\newcommand{\JCMCC}{{\it J. Combin. Math. Combin. Comput.}, }
\newcommand{\GC}{{\it Graphs Combin.}, }

\title{Anti-Ramsey numbers of graphs with some decomposition family sequences
\thanks{This work is supported by the National Natural Science Foundation of China (Nos.11531001 and 11271256),  the Joint NSFC-ISF Research Program (jointly funded by the National Natural Science Foundation of China and the Israel Science Foundation (No. 11561141001)). The first author is supported by the Fundamental Research Funds for the Central Universities (No. WK0010460004) and the China Postdoctoral Science Foundation (No. BH0010000015).
\newline \indent $^{\dagger}$Correspondent author:
Xiao-Dong Zhang (Email: xiaodong@sjtu.edu.cn)}}
\author{ Long-Tu Yuan  \\
{\small School of Mathematical Sciences}\\
{\small University of Science and Technology of China} \\
{\small  96 Jinzhai Road, Hefei, 230026, P.R. China}\\
{\small Email: longtu@ustc.edu.cn}\\
Xiao-Dong Zhang$^{\dagger}$   \\
{\small School of Mathematical Sciences, MOE-LSC, SHL-MAC
}\\
{\small Shanghai Jiao Tong University} \\
{\small  800 Dongchuan Road, Shanghai, 200240, P.R. China}}
\date{}
\maketitle
\begin{abstract}
For a given graph $H$, the anti-Ramsey number of $H$ is the maximum number of colors in an edge-coloring of a complete graph which does not contain a rainbow copy of $H$. In this paper, we extend the decomposition family of graphs to the decomposition family sequence of graphs and show that $K_5$ is determined by its decomposition family sequence. Based on this new graph notation, we determine  the anti-Ramsey numbers for new families of graphs, including the Petersen graph, vertex-disjoint union of cliques, etc., and characterize the extremal colorings.
\end{abstract}

{{\bf Key words:} Anti-Ramsey number; Decomposition family sequence; Progressive induction.}

{{\bf AMS Classifications:} 05C35; 05D99.}
\vskip 0.5cm

\section{Introduction}
\subsection{Basic notations and results}
The graphs considered in this paper are finite, undirected, and simple (no loops or multiple edges). Let $G=(V(G), E(G))$ be a graph, where $V(G)$ is the vertex set with cardinality $v(G)$ and $E(G)$ is the edge set with cardinality $e(G)$.  If $x$ is a vertex of $G$, the {\it neighborhood} of $x$ in $G$ is denoted by $N_G(x)=\{y\in V(G):(x,y)\in E(G)\}$, or when it is clear, simply by $N(x)$. The {\it degree} of $x$ in $G$, denoted by $deg_{G}(x)$, or $d_{G}(x)$, is the size of $N_G(x)$. We use $\delta(G)$ and $\Delta(G)$ to denote the minimum and maximum degrees, respectively, in $G$. For a subset $X\subset V(G)$, let $G[X]$ denote the subgraph of $G$ induced by $X$. Denote by $\overline{G}$ the complement graph of $G$. Denote by $G\cup H$ the vertex disjoint union of $G$ and $H$ and by $k\cdot G$ the vertex disjoint union of $k$ copies of a graph $G$. Denote by $G\vee H$ the graph obtained from $G\cup H$ by adding edges between each vertex of $G$ and each vertex of $H$. The subscript in the case of graphs indicates the number of vertices, e.g., denote by $P_{k}$ a path on $k$ vertices, denote by $C_{k}$ a cycle on $k$ vertices, $S_{k}$ a star on $k$ vertices, $K_{n}$ the complete graph on $n$ vertices, $K_{n_1,\ldots,n_p}$ the complete $p$-partite graph $\overline{K}_{n_1}\vee\ldots\vee \overline{K}_{n_p}$,
denote by $M_{k}$ the disjoint union of $\lfloor\frac{k}{2}\rfloor$ disjoint copies of edges and $\lceil\frac{k}{2}\rceil-\lfloor\frac{k}{2}\rfloor$ isolated vertex.

Let $\mathcal{L}$ be a family of graphs. The {\it Tur\'{a}n number} of $\mathcal{L}$, ex$(n,\mathcal{L})$, is the maximum number of edges in a graph $G$ of order $n$  which does not contain a copy of any $L\in \mathcal{L}$.
In 1941, Tur\'{a}n \cite{turan1941} proved that the extremal graph for $K_{p+1}$ is
the complete $p$-partite graph on $n$ vertices which is balanced, in that the part sizes are as equal as possible (any two sizes differ by at most $1$). This balanced complete $p$-partite graph on $n$ vertices is the {\it Tur\'{a}n graph} $T(n,p)$ and denote by $t(n,p)$ the size of Tur\'{a}n graph $T(n,p)$.
Later, in 1946, Erd\H{o}s and Stone \cite{erdHos1946} proved the following well-know theorem.

\begin{theorem}\cite{erdHos1946}\label{erdos-stone}
For all integers $p\geq 1$ and $N\geq 1$, and every $\epsilon >0$, there exists an integer $n_0$ such that every graph with $n\geq n_0$ vertices and at least
\[t(n,p)+\epsilon n^2\]
edges contains $T(N,p+1)$ as a subgraph.
\end{theorem}
In many ordinary extremal problems the minimum chromatic number plays a decisive role. Let $\mathcal{L}$ be a family of graphs, the {\it subchromatic number} $p(\mathcal{L})$ of $\mathcal{L}$ is defined by
$$p(\mathcal{L})=\min\{\chi(\mathcal{L}):L\in \mathcal{L}\}-1.$$
In 1966, Erd\H{o}s and Simonovits \cite{erdHos1966} proved the following.
\begin{theorem}\cite{erdHos1966}\label{p+1 chromatic}
If $\mathcal{L}$ is a family of graphs with subchromatic number $p>0$, then $$\emph{ex}(n,\mathcal{L})=\left(1-\frac{1}{p}\right){n \choose 2}+o(n^2).$$
\end{theorem}

An {\it  edge-colored graph} is a graph $G=(V(G),E(G))$ with a map $c:E(G)\rightarrow S$. The elements of $S$ are called the {\it colors}. A subgraph of an edge-colored graph is {\it rainbow} (or {\it polychromatic}) if all of its edges have different colors. Let $\mathcal{F}$ be a family of graphs. For the purpose of this paper, we call an edge-coloring of $K_n$ that contains no rainbow copy of any graph in $\mathcal{F}$ an {\it $\mathcal{F}$-free coloring} and call an edge-coloring of $K_n$ a {\it coloring} of $K_n$ for convenience. The {\it anti-Ramsey number} AR$(n,\mathcal{F})$ is the maximum number of colors in an $\mathcal{F}$-free coloring of $K_n$. Anti-Ramsey numbers were introduced by Erd\H{o}s, Simonovits and S\'{o}s \cite{erdossim1975}. Various results about this extremal function have been obtained, see \cite{Alon1983,Alon2015,Axenovich2004,Chen2008,Gorgol2016,Jahan2016,Jiang2002,Jiang2009,Jiang2003,Jiang2004,JJ2006,JJ2002,JJ2005,Sar2005,Sch2004}.

In 1975, Erd\H{o}s, Simonovits and S\'{o}s \cite{erdossim1975} proved the following results.
\begin{theorem}\cite{erdossim1975}
Let $\mathcal{F}$ be an arbitrary graph family with $\mathcal{F}^-=\{H-e:H\in \mathcal{F}, e\in E(H)\}$ and $p(\mathcal{F}^-)=p$. Then
$$\emph{AR}(n,\mathcal{F})=t(n,p)+o(n^2).$$
\end{theorem}
\begin{theorem}\cite{erdossim1975}\label{antiramsey for Kp}
For any $p\geq2$ and all sufficiently large $n$, $n\geq n_0(p)$, we have
$$
\emph{AR}(n,K_{p+2})=t(n,p)+1,
$$
and any coloring achieving this bound is obtained by taking a rainbow $T(n,p)$ and coloring all edges in its complement with the same (extra) color.
\end{theorem}

In order to state the main results in this paper, we introduce the following definitions which are of independent interest.

\subsection{Decomposition family and decomposition family sequence of graphs}
For every family of forbidden graphs $\mathcal{L}$, Simonovits \cite{Simonovits1974} defined the decomposition family $\mathcal{M}(\mathcal{L})$ of $\mathcal{L}$.

\begin{definition}
Given a family $\mathcal{L}$ with $p(\mathcal{L})=p$, let $\mathcal{M}:=\mathcal{M}(\mathcal{L})$ be the family of minimal graphs $M$ that satisfy the following: there exist an $L\in \mathcal{L}$ and a $t=t(L)$ such that $L\subset (M\cup \overline{K}_t)\vee T(t,p-1)$. We call $\mathcal{M}$ the {\it decomposition family} of $\mathcal{L}$.
\end{definition}
Thus, a graph $M$ is in $\mathcal{M}$ if the graph obtained from putting a copy of $M$ (but not any of its proper subgraphs) into a class of a large $T(n,p)$ contains some $L\in \mathcal{L}$. If $L\in\mathcal{ L}$ with minimum chromatic number $p+1$, then $L\subset T(t,p+1)$ for some $t=t(L)$, therefore the decomposition family $\mathcal{M}$ always contains some bipartite graphs. But not all the graphs in the decomposition family of a family of graphs must be bipartite, for example, $\mathcal{M}(\{K_{p+2},2 \cdot K_{p+1}\})=\{K_3,M_4\}$ ($p\geq 1$).

A graph on $2k+1$ vertices consisting of $k$ triangles which intersect in exactly one common vertex is called a {\it $k$-fan}. Similarly, a graph on $pk+1$ vertices consisting of $k$ cliques each with $p+1$ vertices, which intersect in exactly one common vertex, is called a {\it $(k,p+1)$-fan.} It is easy to see that the decomposition families of $k$-fan and $(k,p+1)$-fan are both $\{S_{k+1},M_{2k}\}$.

In this paper, we define the decomposition family  sequence of graphs. Let $\mathcal{F}$ be a family of graphs and $\mathcal{M}(\mathcal{F})$ be the decomposition family of $\mathcal{F}$. The {\it decomposition-remainder family} of $\mathcal{F}$ is defined by
\[\mathcal{F}^{\mathcal{M}(\mathcal{F})-}=\{F-E(M):F\in\mathcal{F},M\in\mathcal{M}(\mathcal{F}),M\subseteq F \}.\]


\begin{definition}
Given a family $\mathcal{F}=\mathcal{F}_0$ with $p(\mathcal{F})=p$, we define $\mathcal{F}_1,\ldots,\mathcal{F}_{p}$ by recurrence:
\[\mathcal{F}_{i+1}=\mathcal{F}_{i}^{\mathcal{M}(\mathcal{F}_{i})-}  \mbox{ and }\mathcal{M}_i(\mathcal{F})=\mathcal{M}(\mathcal{F}_i)\]
 for $i=0,1,\ldots,p-1,$ and call $$\mathcal{M}_0(\mathcal{F}),\ldots,\mathcal{M}_p(\mathcal{F})$$ the decomposition family sequence of $\mathcal{F}$.
\end{definition}

Let $\mathcal{L}$ be a family of graphs. The decomposition family of $\mathcal{L}$  does not determine ex$(n,\mathcal{L})$, see \cite{Yuan1} for more information. Hence, it is interesting to ask the following questions.

\begin{question}
Are the extremal graphs for a family of graphs $\mathcal{L}$ determined by its decomposition family sequence?
\end{question}

\begin{question}
Is a family of  graphs determined by its decomposition family sequence?
\end{question}

In the conclusion of this paper, we will discuss some problems on decomposition family sequences of graphs.

\subsection{Main results}
Let $H(n,p,k)=K_{k-1}\vee T(n-k+1,p)$ and $H^\prime(n,p,k)=\overline{K}_{k-1}\vee T(n-k+1,p)$. Denote by $h(n,p,k)$ the size of $H(n,p,k)$ and $h^\prime(n,p,k)$ the size of $H^\prime(n,p,k)$.\\

One of the main results of this paper is the following.

\begin{theorem}\label{main1}
Let $\mathcal{F}$ be a family of graphs with $\mathcal{F}^{-}=\{H-e:H\in \mathcal{F}, e\in E(H)\}$ and $p(\mathcal{F}^{-})=p\geq 2$. Let $\mathcal{M}_0(\mathcal{F}),\ldots,\mathcal{M}_p(\mathcal{F})$ be the decomposition family sequence of $\mathcal{F}$.\\
(\romannumeral1) If $\mathcal{M}_0(\mathcal{F})=\{M_{2k}\}$ ($k\geq2$), then $$\emph{AR}(n,\mathcal{F})=h^\prime(n,p,k-1)+q$$
provided $n$ is sufficiently large, where $q$ is the maximum number of colors such that the coloring of $K_n$, obtained by taking a rainbow $H^\prime(n,p,k-1)$, and coloring all edges in its complement with $q$ (extra) colors so that in each partite set of $T(n-k+2,p)$ we color all edges with the same color, is an $\mathcal{F}$-free coloring.\\
(\romannumeral2) If $\mathcal{M}_0(\mathcal{F})=\{M_2\}$, $\mathcal{M}_1(\mathcal{F})=\{M_{2k-2}\}$ ($k\geq2$), then $$\emph{AR}(n,\mathcal{F})=h(n,p,k-1)+1$$
provided $n$ is sufficiently large.
Moreover, any coloring achieving this bound is obtained by taking a rainbow $H(n,p,k-1)$, and coloring all edges in its complement with one (extra) color.
\end{theorem}
\remark  Clearly Theorem~\ref{main1} (\romannumeral2) generalized   Theorem~\ref{antiramsey for Kp}. 
 The value $q$ in Theorem~\ref{main1} (\romannumeral1) depends on the structures of the graphs in $\mathcal{F}$. For example, if $\mathcal{F}$ contains the only graph obtained by taking a Tur\'{a}n graph $T(n,p-1)$ with $n\geq 6(p-1)$ and putting a $2\cdot K_3$ in one partite set of it, then $q=1$. If $\mathcal{F}$ contains the only graph obtained by taking a Tur\'{a}n graph $T(n,p)$ with $n\geq 4p$ and putting an $M_4$ in one partite set of it, then $q=p$.
\begin{corollary}
 Let $\mathcal{F}=\{k\cdot K_{p+1}\}$, $p\geq 2$, $k\geq 2$. Then $$\emph{AR}(n,\mathcal{F})=h^\prime(n,p,k-1)+{ k-2 \choose 2}+1$$
provided $n$ is sufficiently large.
Moreover, any coloring achieving this bound is obtained by taking a rainbow $H(n,p,k-1)$, and coloring all edges in its complement with one (extra) color.
\end{corollary}

\begin{corollary}\label{corollary1}
 Let $\mathbb{P}_{10}$ be the Petersen graph. Then $$\emph{AR}(n,\{\mathbb{P}_{10}\})=\left\lfloor\frac{n-1}{2}\right\rfloor
 \left\lceil\frac{n-1}{2}\right\rceil+n+1$$
provided $n$ is sufficiently large.
Moreover, any coloring achieving this bound is obtained by taking a rainbow $H(n,2,2)$ and coloring all edges in its complement with $2$ (extra) colors so that in each component of $\overline{H}(n,2,2)$ we color all edges with the same color.
\end{corollary}
\begin{proof}
Since $\mathcal{M}_0(\mathbb{P}_{10})=\{M_6\}$, the corollary follows from Theorem~\ref{main1} (\romannumeral1) with $q=2$.
\end{proof}

A {\it nearly $(k-1)$-regular} graph is a graph such that any vertex of it has degree $k-1$ except one vertex with degree $k-2$. The following proposition was proved in \cite{Simonovits1968}.

\begin{proposition}\cite{Simonovits1968}\label{k-1 regular}
Let $m$ be a large constant. Then there exists a $(k-1)$-regular triangle-free graph or a nearly $(k-1)$-regular triangle-free graph on $m$ vertices.
\end{proposition}

Denote by $\mathcal{U}_{n,k} $ the class of $(k-1)$-regular graphs, or nearly $(k-1)$-regular graphs on $n$ vertices. Let $\mathcal{U}_{n,k}^{\prime}$ be the class of $(k-1)$-regular triangle-free graphs, or nearly $(k-1)$-regular triangle-free graphs on $n$ vertices. By Proposition~\ref{k-1 regular}, $\mathcal{U}_{n,k}$ and $\mathcal{U}_{n,k}^{\prime}$ are nonempty for each $k$ and large $n$.

Let $n=\sum^{p}_{i=1}n^\prime_i$ with $n^\prime_1\geq \ldots \geq n^\prime_p$ and $n^\prime_1-n^\prime_p\leq 1$. Denote by $T(n,p;\mathcal{U}_{n,k})$ ($T(n,p;\mathcal{U}^\prime_{n,k})$ resp.) the class of graphs obtained from the Tur\'{a}n graph $T(n,p)$ by adding a graph of $\mathcal{U}_{n^\prime_i,k}$ ($\mathcal{U}_{n^\prime_i,k}^{\prime}$ resp.) into the $i$-th class of it for $i=1,\ldots,p$. Let $n=\sum^{p}_{i=1}n_i$. Similarly, we can define $K_{n_1,n_2,\ldots,n_p}(n;\mathcal{U}_{n,k})$ and $K_{n_1,n_2,\ldots,n_p}(n;\mathcal{U}^\prime_{n,k})$. Denote by $T(n,p;\mathcal{U}_{\lceil\frac{n}{p}\rceil,k})$ the class of graphs obtained from the Tur\'{a}n graph $T(n,p)$ by adding a graph from $\mathcal{U}_{\lceil\frac{n}{p}\rceil,k}$ into the largest class of it. Similarly, we can define $K_{n_1,n_2,\ldots,n_p}(n;\mathcal{U}_{n_1,k})$.



The following colorings of graphs are the extremal colorings of the next theorem. Let $n_1\geq \ldots \geq n_p$. Let $\mathcal{C}(n,k,p)$ ($\mathcal{C}^\prime(n,k,p)$ resp,) be the set of colorings of $K_n$  obtained by taking a rainbow coloring of a graph in $T(n,p;\mathcal{U}_{n,k})$ ($T(n,p;\mathcal{U}^\prime_{n,k})$ resp,) or of a graph in $K_{n_1,n_2,\ldots,n_p}(n;\mathcal{U}_{n,k})$ ($K_{n_1,n_2,\ldots,n_p}(n;\mathcal{U}^\prime_{n,k})$ resp,) with $n_1-n_p=2$, and $n_1, n_p$ are even when $k$ is even, respectively, and coloring all the other edges in its partite sets with $p$ (extra) colors so that in each partite set we color all edges with the same color.

Let $Q(p,k)=K_1\vee T(pk,p)$. The other main result is stated as follows. 

\begin{theorem}\label{main2}
Let $\mathcal{F}$ be a family of graphs with $\mathcal{F}^{-}=\{H-e:H\in \mathcal{F}, e\in E(H)\}$ and $p(\mathcal{F}^{-})=p\geq 2$. Let $\mathcal{M}_0(\mathcal{F}),\ldots,\mathcal{M}_p(\mathcal{F})$ be the decomposition family sequence of $\mathcal{F}$.\\
(\romannumeral1) Let $k\geq 2$. If $\mathcal{M}_0(\mathcal{F})=\{S_{k+1}\}$ and any graph in $\mathcal{F}_0$ contains $Q(p,k)$ as a subgraph, then $$\emph{AR}(n,\mathcal{F})=t(n,p)+\left\lfloor\frac{(k-2)\lceil\frac{n}{p}\rceil}{2}\right\rfloor+\ldots+\left\lfloor\frac{(k-2)\lfloor\frac{n}{p}\rfloor}{2}\right\rfloor+p$$
provided $n$ is sufficiently large. Furthermore, the colorings in $\mathcal{C}^\prime(n,k-1,p)$ are extremal colorings, and all extremal colorings are in $\mathcal{C}(n,k-1,p)$.\\
(\romannumeral2) Let $k\geq 3$. If $\mathcal{M}_0(\mathcal{F})=\{M_2\}$, $\mathcal{M}_1(\mathcal{F})=\{S_{k+1}\}$ and any graph in $\mathcal{F}_1$ contains $Q(p,k)$ as a subgraph, then
$$\emph{AR}(n,\mathcal{F})=t(n,p)+\left\lfloor\frac{(k-2)\lceil\frac{n}{p}\rceil}{2}\right\rfloor+\ldots+\left\lfloor\frac{(k-2)\lfloor\frac{n}{p}\rfloor}{2}\right\rfloor+p .$$
provided $n$ is sufficiently large. Furthermore, the coloring in $\mathcal{C}^\prime(n,k-1,p)$ are extremal colorings, and all extremal colorings are in $\mathcal{C}(n,k-1,p)$ and except the extremal colorings which is obtained by taking a rainbow graph from  $T(n,2;\mathcal{U}_{n/2,3})$ and coloring all edges in its complement with one (extra) color, when $k=3$, $p=2$ and $n=2$ mod $4$.
\end{theorem}

\remark If $\mathcal{F}$ is a family of graphs with $p(\mathcal{F}^-)= p$ and there exist a graph $H\in \mathcal{F}$ and two edges $e_1,e_2\in E(H)$ such that $\chi(H-e_1-e_2)=p$. Jiang and Pikhurko \cite{Jiang2009} determined the anti-Ramsey number of $\mathcal{F}$. Theorems~\ref{main1} and \ref{main2} generalize their results in a certain sense.

The rest of this paper is organised as follows: In Section 2.1, lemma of progressive induction is presented.
In Section 2.2, serval lemmas are presented. In Section 3,  the proof of the main theorems are given. In Section 4,  some problems on  decomposition family sequence of graphs are discussed.

\section{Several lemmas}
\subsection{Lemma of progressive induction.}
In 1960s, Simonovits \cite{Simonovits1968} introduced the so-called {\it progressive induction} which is similar to the mathematical induction and Euclidean algorithm and combined from them in a certain sense. The progressive induction method is key 
powerful for extremal problems of non-bipartite graphs, for example see \cite{YuanJGT}.
\begin{lemma}\cite{Simonovits1968}\label{progrssion induction}
Let $\mathfrak{U}=\cup_{1}^{\infty}\mathfrak{U}_n$ be a set of given elements such that $\mathfrak{U}_n$ are disjoint subsets of $\mathfrak{U}$. Let $B$ be a condition or property defined on $\mathfrak{U}$ (i.e. the elements of $\mathfrak{U}$ may satisfy or not satisfy $B$). Let $\phi$ be a function from $\mathfrak{U}$ to non-negative integers and\\
(a) if $a\in \mathfrak{U}$ satisfies $B$, then $\phi(a)=0$.\\
(b) there is an $M_0$ such that if $n>M_0$ and $a\in \mathfrak{U}_n$ then either $a$ satisfies $B$ or there exist an $n^{\prime}$ and an $a^{\prime}$ such that
\[\frac{n}{2}<n^{\prime}<n, a^{\prime}\in \mathfrak{U}_{n^{\prime}} \mbox{ and } \phi(a)<\phi(a^{\prime}).\]
Then there exists an $n_0$ such that if $n>n_0$, from $a\in \mathfrak{U}_n$ follows  that $a$ satisfies $B$.
\end{lemma}

\remark In our problems, $\mathfrak{U}_n$ is the set of extremal $\mathcal{F}$-free colorings of $K_n$, $B$ is a property that the coloring of $K_n$ belongs to the coloring sets described in Theorems~\ref{main1} or \ref{main2}.

\subsection{Other lemmas}
The following two lemmas are proved in \cite{Simonovits1968}.
\begin{lemma}\label{lemma1}\cite{Simonovits1968}
Let $G_n$ be a graph on $n$ vertices. If $\chi(G_n)=p$ and $A_1,\ldots,A_p$ are the sets of vertices having the $i$-th color at a fixed coloring of $V(G_n)$ with $p$ colors and $m_i$ is the number of vertices of the $i$-th class of $T(n,p)$ (i.e. $m_i=\lceil n/p\rceil$ or $m_i=\lfloor n/p\rfloor$ and $\sum_{i=1}^{p}m_i=n$), furthermore $|A_i|=m_i+s_i$, then $$e(G_n)\leq t(n,p)-\sum_{i=1}^{p}{|s_i| \choose 2}.$$
\end{lemma}

\begin{lemma}\label{lemma2}\cite{Simonovits1968}
Let $\mathcal{F}$ be a family of graphs with $\mathcal{M}_0(\mathcal{F})=\{S_{k+1}\}$ and $p(\mathcal{F})=p$. If each $F\in\mathcal{F}$ contains $Q(p,k)$ as a subgraph, then each graph in $T(n,p;\mathcal{U}^\prime_{n,k})$ does not contain any $F\in \mathcal{F}$ as a subgraph.
\end{lemma}

The following definition is a key notion in the proof of our lemma and theorems.
\begin{definition}
The representing graph, denoted by $L_n$, of a coloring $c_n$  is a spanning subgraph of $K_n$ obtained by taking one edge of each color in $c_n$ (where $L_n$ may contain isolated vertices).
\end{definition}

We need the following lemma to show that the colorings in Theorem~\ref{main2} are $\mathcal{F}$-free.

\begin{lemma}\label{lemma3}
Let $\mathcal{F}$ be a family of graphs and $k\geq 2$. If each $F\in\mathcal{F}$ contains $Q(p,k)$ as a subgraph, then the colorings in $\mathcal{C}^\prime(n,k-1,p)$ are $\mathcal{F}$-free.
\end{lemma}
\begin{proof} It is sufficient to show that each coloring in $\mathcal{C}^\prime(n,k-1,p)$ does not contain a rainbow $Q(p,k)$ as a subgraph. If $k\leq3$, then the representing graph $L_n$ of a coloring in $\mathcal{C}^\prime(n,k-1,p)$  is a subgraph of some graph in $T(n,p;\mathcal{U}^\prime_{n,k})$. Hence  by Lemma~\ref{lemma2}, $L_n$ does not contain any $F\in \mathcal{F}$ as a subgraph. Thus we finish the proof of the lemma for $k\leq 3$.

Let $k\geq4$, we will prove this lemma by applying mathematical induction on $p$. If $p=1$, then the representing graph $L_n$ of a coloring in $\mathcal{C}^\prime(n,k-1,p)$ is a graph with $\Delta(L_n)=k-1$. Since $Q(1,k)=S_{k+1}$, the lemma is obviously true. It will be shown that if the statement is not true for $p$, then it is not true for some $t\leq p-1$. This implies the lemma.

Suppose that there is a coloring $c_n$ in $\mathcal{C}^\prime(n,k-1,p)$ containing a rainbow $Q(p,k)$ as a subgraph.  We partition $V(Q(p,k))$ into $\{q\}\cup Q_1\cup \ldots \cup Q_p$ so that there is no edge in $Q(p,k)[Q_i]$ and $|Q_i|=k$ for $i=1,\ldots,p$. Let $L_n$ be a representing graph of $c_n$ which contains $Q(p,k)$ as a subgraph. Let $V_1\cup\ldots\cup V_p$ be a vertex partition of $L_n$ such that after removing an edge $x_iy_i$ of $L_n[V_i]$, the resulting graph is a $(k-2)$-regular triangle-free graph or a nearly $(k-2)$-regular triangle-free graph for $i=1,\ldots,p$. We will show that there is a $Q(t,k)$ whose vertices lie in $t$ partite sets of $L_n$ for some $t\leq p-1$, or we can not find all vertices of $Q(p,k)$ in $L_n$. This proves the lemma.\\

{\bf Observation 1.} Any $k-1$ vertices of $L_n[V_i]$ have at most one common neighbour in $L_n[V_i]$. In other words, in $L_n[V_i]$, the common neighbours of  any two vertices are at most $k-2$.\\

\begin{proof} It follows from that $x_i$ and $y_i$ are the only two vertices with degree $k-1$ in $L_n[V_i]$ and they are joint to each other.\end{proof}

Without loss of generality, let $q\in V_1$ and $q\neq y_1$. Since $L_n[V_1]-x_1y_1$ is a triangle-free graph, there is no edge in $L_n[N_{L_n[V_1]}(q)\setminus \{y_1\}]$. Hence $N_{L_n[V_1]}(q)\setminus \{y_1\}$ belongs to at most one partite set, say $Q_1$, of $Q(p,k)$.  Suppose that  $|N_{L_n[V_1]}(q)\setminus \{y_1\}\cap Q_1|\geq 1$ and $|N_{L_n[V_1]}(q)\setminus \{y_1\}\cap (\cup_{i=2}^pQ_i)|=0$. Indeed, if $N_{L_n[V_1]}(q)\setminus \{y_1\}\cap V(Q(p,k))=\emptyset$, then $L_n[\cup_{i=2}^{p}V_i]$ contains
$$Q(p-1,k)\subseteq T(pk-k,p-1)\vee \overline{K}_{k-1}$$
 as a subgraph. Hence the statement is not true for $p-1$, and we are done.

If $y_1$ does not belong to $\cup^{p}_{i=2}Q_{i}$, then $L_n[\cup_{i=2}^{p}V_i]$ contains $Q(p-1,k)$ as a subgraph, we are done. Without loss of generality, let $y_1$ belongs to $Q_{p}$ ($y_1$ must be joint to $q$). Since $|N_{L_n[V_1]}(q)\setminus \{y_1\}|\leq k-2$,   there are at least two vertices of $Q_1$ which do not belong to $V_1$. Let $Q^{\prime}_1= Q_1-N_{L_n[V_1]}(x)$ and $|Q^{\prime}_1|=\ell\geq 2$. There is no edge in $L_n[Q^{\prime}_1]$, otherwise $L_n[\cup_{i=2}^{p}V_i]$ contains $$Q(p-1,k)\subseteq(K_{2}\cup \overline{K}_{\ell-2})\vee T(pk-2k,p-2)\vee\overline{K}_{k-1}$$ as a subgraph, we are done.

Actually, we have proved the following claim.\\

{\bf Claim 1.} For $t\in \{2,\ldots,p\}$, let $Q^\prime_{t-1}=Q_{t-1}\cap (\cup_{i=t}^{p}V_i)$. If $(\cup_{i=t}^{p}Q_{i})\cap (\cup_{i=1}^{t-1}V_i)=\{y_1\}$ and $|Q_{t-1}^\prime|\geq 2$, then $e(L_n[Q_{t-1}^\prime])=0$.\\

\begin{proof}
Otherwise, $L_n[\cup_{i=t}^{p}V_i]$ contains $$Q(p-t+1,k)\subseteq K_{2} \vee T(pk-tk+k,p-t+1)\vee\overline{K}_{k-1}$$ as a subgraph. Thus the claim holds.
\end{proof}

\begin{center}
\begin{tikzpicture}[scale = 0.5]

\filldraw[fill=black] (-10,2.5) circle(2pt);
\filldraw[fill=black] (-10,1.5) circle(2pt);
\draw [line width=1pt](-10.5,-2) -- (-10.5,-0.5);
\draw [line width=1pt](-9.5,-2) -- (-9.5,-0.5);
\draw [line width=1pt](-9.5,-2) -- (-10.5,-2);
\draw [line width=1pt](-9.5,-0.5) -- (-10.5,-0.5);
\draw (-10,0) ellipse (1.2 and 4);
\draw (-7,0) ellipse (1.2 and 4);
\filldraw[fill=black] (-4.5,0) circle(2pt);
\filldraw[fill=black] (-3.5,0) circle(2pt);
\filldraw[fill=black] (-2.5,0) circle(2pt);
\draw (0,0) ellipse (1.2 and 4);
\draw (3,0) ellipse (1.2 and 4);
\draw (6,0) ellipse (1.2 and 4);
\draw (9,0) ellipse (1.2 and 4);
\filldraw[fill=black] (11.5,0) circle(2pt);
\filldraw[fill=black] (12.5,0) circle(2pt);
\filldraw[fill=black] (13.5,0) circle(2pt);
\draw (16,0) ellipse (1.2 and 4);

\draw node at (-10,-5){$V_1$};
\draw node at (-7,-5){$V_2$};
\draw node at (0,-5){$V_{t-1}$};
\draw [line width=1pt](-0.5,2.5) -- (-0.5,-1.5);
\draw [line width=1pt](0.5,2.5) -- (0.5,-1.5);
\draw [line width=1pt](0.5,2.5) -- (-0.5,2.5);
\draw [line width=1pt](0.5,-1.5) -- (-0.5,-1.5);
\draw node at (0,-2){$Q_{t-1}$};

\draw node at (3,-5){$V_{t}$};
\draw [line width=1pt](2.5,2.5) -- (2.5,1.5);
\draw [line width=1pt](3.5,2.5) -- (3.5,1.5);
\draw [line width=1pt](3.5,2.5) -- (2.5,2.5);
\draw [line width=1pt](3.5,1.5) -- (2.5,1.5);
\draw node at (3,1){$Q_{t-1}$};
\draw [line width=1pt](2.5,0.5) -- (2.5,-0.5);
\draw [line width=1pt](3.5,0.5) -- (3.5,-0.5);
\draw [line width=1pt](3.5,0.5) -- (2.5,0.5);
\draw [line width=1pt](3.5,-0.5) -- (2.5,-0.5);
\draw node at (3,-0.2){$x_{t}$};
\filldraw[fill=black] (3,0.2) circle(2pt);
\draw node at (3,-1){$Q_{t}$};
\draw [line width=1pt](2.5,-2.5) -- (2.5,-1.5);
\draw [line width=1pt](3.5,-2.5) -- (3.5,-1.5);
\draw [line width=1pt](3.5,-2.5) -- (2.5,-2.5);
\draw [line width=1pt](3.5,-1.5) -- (2.5,-1.5);
\draw node at (3,-3){$Q_{t+1}$};
\filldraw[fill=black] (3,-1.8) circle(2pt);
\draw node at (3,-2.2){$y_{t}$};

\draw node at (6,-5){$V_{t+1}$};
\draw [line width=1pt](5.5,2.5) -- (5.5,-1.5);
\draw [line width=1pt](6.5,2.5) -- (6.5,-1.5);
\draw [line width=1pt](6.5,2.5) -- (5.5,2.5);
\draw [line width=1pt](6.5,-1.5) -- (5.5,-1.5);
\draw node at (6,-2){$Q_{t}$};

\draw node at (9,-5){$V_{t+2}$};
\draw [line width=1pt](8.5,2.5) -- (8.5,-1.5);
\draw [line width=1pt](9.5,2.5) -- (9.5,-1.5);
\draw [line width=1pt](9.5,2.5) -- (8.5,2.5);
\draw [line width=1pt](9.5,-1.5) -- (8.5,-1.5);
\draw node at (9,-2){$Q_{t+1}$};

\draw node at (16,-5){$V_{p}$};
\draw node at (-10,-2.5){$Q_1$};
\draw node at (-10,1){$y_1$};
\draw node at (-10,2){$q$};
\draw [line width=1pt](-10.5,1.7) -- (-10.5,0.7);
\draw [line width=1pt](-9.5,1.7) -- (-9.5,0.7);
\draw [line width=1pt](-9.5,1.7) -- (-10.5,1.7);
\draw [line width=1pt](-9.5,0.7) -- (-10.5,0.7);
\draw node at (-10,0){$Q_p$};

\draw [line width=1pt](-7.5,-0.5) -- (-7.5,1);
\draw [line width=1pt](-6.5,-0.5) -- (-6.5,1);
\draw [line width=1pt](-6.5,-0.5) -- (-7.5,-0.5);
\draw [line width=1pt](-6.5,1) -- (-7.5,1);
\draw node at (-7,-1){$Q^\prime_1$};

\draw node at (3,-7){Figure 1};

\end{tikzpicture}
\end{center}

{\bf Observation 2.} $(P_3\cup \overline{K}_{k-4})\vee(K_2\cup \overline{K}_{k-3})\vee\overline{K}_{k-1} $ contains $Q(3,k)$ as a subgraph.\\

{\bf Observation 3.} $(K_3\cup \overline{K}_{k-4})\vee T(2k-2,2)$ contains $Q(2,k)$ as a subgraph.\\

 {\bf Claim 2.} For $t\in \{2,\ldots,p-2\}$, if  $(\cup_{i=t}^{p}Q_{i})\cap (\cup_{i=1}^{t-1}V_i)=\{y_1\}$ and $|Q_{t-1}\cap V_{t}|\geq 2$, then at most one of $Q_{t}\cap V_t,\ldots,Q_p\cap V_t$ is non-empty.\\

\begin{proof}
Suppose that the claim does not hold. Then there are exact two  sets among $Q_{t}\cap V_t,\ldots,Q_p\cap V_t$ which are non-empty, otherwise $L_n[V_t]$ contains $K_4$ as a subgraph, contradicting the fact that $L_n[V_t]-x_ty_t$ is a triangle-free graph. We will prove the claim in the following two cases.\\

{\bf Case 1.} $Q_p\cap V_t=\emptyset.$ (Figure 1)\\

Without loss of generality, let $Q_{t}\cap V_t$, $Q_{t+1}\cap V_t$ be non-empty sets and $(\cup_{i=t+2}^{p}Q_{i})\cap V_{t}=\emptyset$. Since any triangle in $L_n[V_{t}]$ contains the edge  $x_ty_t$ and $|Q_{t-1}\cap V_{t}|\geq 2$, we have that, without loss of generality, $x_t$ is the only vertex of $V_t$ which belongs to $Q_t$ and $y_t$ is the only vertex of $V_t$ which belongs to $Q_{t+1}$. Hence, we have
\begin{equation}\label{eq1 for case1}
|Q_t\cap (\cup_{i=t+1}^p V_i) |=|Q_{t+1}\cap (\cup_{i=t+1}^p V_i) |=|Q_p\cap (\cup_{i=t+1}^p V_i) |=k-1.
\end{equation}
Therefore $L_n[\cup_{i=t+1}^p V_i]$ contains $$\overline{K}_{k-1}\vee \overline{K}_{k-1}\vee \overline{K}_{k-1}\vee T(pk-tk-2k,p-t-2)$$ as a subgraph. Thus if there are at least two sets among $Q_{t},Q_{t+1},Q_{p}$ such that the vertices of each of them contains at least one vertex from each of two of $V_{t+1},\ldots,V_{p}$, then, by Observation 2, $L_n[\cup_{i=t+1}^p V_i]$ contains $$Q(p-t,k)\subseteq(P_3\cup \overline{K}_{k-4})\vee(K_2\cup \overline{K}_{k-3})\vee \overline{K}_{k-1}\vee T(pk-tk-2k,p-t-2)$$
as a subgraph, we are done. Now, without loss of generality, suppose that all those $k-1$ vertices of $Q_t$ are in $V_{t+1}$ and all those $k-1$ vertices of $Q_{t+1}$ are in $V_{t+2}$. Since  $\Delta(L_n[V_i]-x_iy_i)\leq k-2$,  by Observation 1, there is at most one vertex of $Q_{t+2}\cup\ldots\cup Q_p$ in $V_i$ for $i=t+1,t+2$.

For $i\in \{t+1,t+2\}$, if there is no vertex of $Q_{t+2}\cup\ldots\cup Q_p$ in $V_{i}$, then  $L_n[\cup_{j\neq i}^{p}V_j]$ contains
$Q(p-1,k)$ as a subgraph, we are done. By Observation 1, we may assume that there are exact one vertex of $Q_{t+2}\cup\ldots\cup Q_p$ in $V_{t+1}$ and exact one vertex of $Q_{t+2}\cup\ldots\cup Q_p$ in $V_{t+2}$. Moreover, at most one of $Q_p\cap V_{t+1}$ and $Q_p\cap V_{t+2}$ is non-empty set. Otherwise, since $k\geq4$, by Observation 3, $L_n[\cup_{i=t+1}^{p}V_i]$ contains
$$Q(p-t,k)\subseteq(K_3\cup \overline{K}_{k-4})\vee T(2k-2,2)\vee T(pk-tk-2k,p-t-2)$$
as a subgraph, and we are done. Thus, without loss of generality, there is exactly one vertex of $Q_{t+2}$ which belongs to $V_{t+1}$. Thus,  we have

\begin{equation}\label{eq2 for case1}
|Q_{t+1}\cap (\cup_{i=t+2}^p V_i) |=|Q_{t+2}\cap (\cup_{i=t+2}^p V_i) |=|Q_p\cap (\cup_{i=t+2}^p V_i) |=k-1.
\end{equation}
Repeat the previous proof $p-t-3$ times (from (\ref{eq1 for case1}) to (\ref{eq2 for case1})), we have
$$|Q_{p-2}\cap (V_{p-1}\cup V_p) |=|Q_{p-1}\cap (V_{p-1}\cup V_p) |=|Q_p\cap (V_{p-1}\cup V_p) |=k-1.$$
Moreover, without loss of generality, we have $|Q_{p-2}\cap V_{p-1}|=|Q_{p-1}\cap V_{p}|=k-1$. Otherwise there are two of $Q_{p-2}$, $Q_{p-1}$, $Q_{p}$ such that the vertices of each of them contains vertices from both of $V_{p-1}$ and $V_{p}$. Hence $L_n[V_{p-1}\cup V_p]$ contains
$$Q(2,k)\subseteq (P_3\cup \overline{K}_{k-4})\vee(K_2\cup \overline{K}_{k-3})\vee \overline{K}_{k-1}$$
 as a subgraph, we are done. Thus, by Observation 1, we have $|Q_{p}\cap (V_{p-1}\cup V_p)|=|Q_{p}\cap V_p|+|Q_{p}\cap V_{p-1}|\leq 2,$ contradicting $|Q_p\cap (V_{p-1}\cup V_p) |=k-1\geq 3$ ($k\geq 4$).\\

{\bf Case 2.} $Q_p\cap V_t\neq \emptyset.$ (Figure 2)\\

Without loss of generality, let $Q_{t}\cap V_t\neq \emptyset$ and $(\cup_{i=t+1}^{p-1}Q_{i})\cap V_{t}=\emptyset$. There exists a $j\in\{t+1,\ldots,p\}$ such that  $|Q_{t}\cap V_{j}|=k-1$. Otherwise, $L_n[\cup_{i=t+1}^{p} V_i]$ contains $$Q(p-t,k)\subseteq (\overline{K}_{k-4}\cup P_3)\vee \overline{K}_{k-2}\vee T(pk-tk-k,p-t-1)$$ a subgraph. Moreover, there exists a $j^\prime\in\{t+1,\ldots,p\}$ such that  $|Q_{p}\cap V_{j^\prime}|=k-2$, otherwise, $L_n[\cup_{i=t+1}^{p} V_i]$ contains $$Q(k,p-t)\subseteq (\overline{K}_{k-4}\cup K_2)\vee \overline{K}_{k-1}\vee T(pk-tk-k,p-t-1)$$ as a subgraph. By Observation 1, we have $j\neq j^\prime$. Without loss of generality, let $|Q_{t}\cap V_{t+1}|=k-1$ and $|Q_{p}\cap V_{p}|=k-2$. Hence, by Observation 1, there is at most one vertex of $\cup_{i=t+2}^{p}Q_i$ belongs to $V_{t+1}$. There is exactly one vertex of $\cup_{i=t+2}^{p}Q_i$ belongs to $V_{t+1}$. Otherwise $L_n[\cup_{i\neq t+1}^{p}V_{i}]$ contains $Q(p-1,k)$ as a subgraph,  and we are done. Hence, without loss of generality, we may suppose that $|Q_{t+1}\cap V_{t+1}|=1$. We may go on this procedure and get  $|Q_{p-1}\cap V_{p}|=k-1$. Hence, by Observation 1, we have $|Q_{p}\cap V_{p}|\leq 1$, a contradiction to $|Q_{p}\cap V_{p}|=k-2\geq 2$.\end{proof}

\begin{center}
\begin{tikzpicture}[scale = 0.5]

\filldraw[fill=black] (-10,2.5) circle(2pt);
\filldraw[fill=black] (-10,1.5) circle(2pt);
\draw [line width=1pt](-10.5,-2) -- (-10.5,-0.5);
\draw [line width=1pt](-9.5,-2) -- (-9.5,-0.5);
\draw [line width=1pt](-9.5,-2) -- (-10.5,-2);
\draw [line width=1pt](-9.5,-0.5) -- (-10.5,-0.5);
\draw (-10,0) ellipse (1.2 and 4);
\draw (-7,0) ellipse (1.2 and 4);
\filldraw[fill=black] (-4.5,0) circle(2pt);
\filldraw[fill=black] (-3.5,0) circle(2pt);
\filldraw[fill=black] (-2.5,0) circle(2pt);
\draw (0,0) ellipse (1.2 and 4);
\draw (3,0) ellipse (1.2 and 4);
\draw (6,0) ellipse (1.2 and 4);
\draw (9,0) ellipse (1.2 and 4);
\filldraw[fill=black] (11.5,0) circle(2pt);
\filldraw[fill=black] (12.5,0) circle(2pt);
\filldraw[fill=black] (13.5,0) circle(2pt);
\draw (16,0) ellipse (1.2 and 4);

\draw node at (-10,-5){$V_1$};
\draw node at (-7,-5){$V_2$};
\draw node at (0,-5){$V_{t-1}$};
\draw [line width=1pt](-0.5,2.5) -- (-0.5,-0.5);
\draw [line width=1pt](0.5,2.5) -- (0.5,-0.5);
\draw [line width=1pt](0.5,2.5) -- (-0.5,2.5);
\draw [line width=1pt](0.5,-0.5) -- (-0.5,-0.5);
\draw node at (0,-1){$Q_{t-1}$};

\draw node at (3,-5){$V_{t}$};
\draw [line width=1pt](2.5,2.5) -- (2.5,1.5);
\draw [line width=1pt](3.5,2.5) -- (3.5,1.5);
\draw [line width=1pt](3.5,2.5) -- (2.5,2.5);
\draw [line width=1pt](3.5,1.5) -- (2.5,1.5);
\draw node at (3,1){$Q_{t-1}$};
\draw [line width=1pt](2.5,0.5) -- (2.5,-0.5);
\draw [line width=1pt](3.5,0.5) -- (3.5,-0.5);
\draw [line width=1pt](3.5,0.5) -- (2.5,0.5);
\draw [line width=1pt](3.5,-0.5) -- (2.5,-0.5);
\draw node at (3,-0.2){$x_{t}$};
\filldraw[fill=black] (3,0.2) circle(2pt);
\draw node at (3,-1){$Q_{t}$};
\draw [line width=1pt](2.5,-2.5) -- (2.5,-1.5);
\draw [line width=1pt](3.5,-2.5) -- (3.5,-1.5);
\draw [line width=1pt](3.5,-2.5) -- (2.5,-2.5);
\draw [line width=1pt](3.5,-1.5) -- (2.5,-1.5);
\draw node at (3,-3){$Q_{p}$};
\filldraw[fill=black] (3,-1.8) circle(2pt);
\draw node at (3,-2.2){$y_{t}$};

\draw node at (6,-5){$V_{t+1}$};
\draw [line width=1pt](5.5,2.5) -- (5.5,-0.5);
\draw [line width=1pt](6.5,2.5) -- (6.5,-0.5);
\draw [line width=1pt](6.5,2.5) -- (5.5,2.5);
\draw [line width=1pt](6.5,-0.5) -- (5.5,-0.5);
\draw node at (6,-1){$Q_{t}$};
\draw [line width=1pt](5.5,-1.5) -- (5.5,-2.5);
\draw [line width=1pt](6.5,-1.5) -- (6.5,-2.5);
\draw [line width=1pt](6.5,-1.5) -- (5.5,-1.5);
\draw [line width=1pt](6.5,-2.5) -- (5.5,-2.5);
\draw node at (6,-3){$Q_{t+1}$};

\draw node at (9,-5){$V_{t+2}$};

\draw node at (16,-5){$V_{p}$};
\draw [line width=1pt](15.5,2.5) -- (15.5,-1.5);
\draw [line width=1pt](16.5,2.5) -- (16.5,-1.5);
\draw [line width=1pt](16.5,2.5) -- (15.5,2.5);
\draw [line width=1pt](16.5,-1.5) -- (15.5,-1.5);
\draw node at (16,-2){$Q_{p}$};

\draw node at (-10,-2.5){$Q_1$};
\draw node at (-10,1){$y_1$};
\draw node at (-10,2){$q$};
\draw [line width=1pt](-10.5,1.7) -- (-10.5,0.7);
\draw [line width=1pt](-9.5,1.7) -- (-9.5,0.7);
\draw [line width=1pt](-9.5,1.7) -- (-10.5,1.7);
\draw [line width=1pt](-9.5,0.7) -- (-10.5,0.7);
\draw node at (-10,0){$Q_p$};

\draw [line width=1pt](-7.5,-0.5) -- (-7.5,1);
\draw [line width=1pt](-6.5,-0.5) -- (-6.5,1);
\draw [line width=1pt](-6.5,-0.5) -- (-7.5,-0.5);
\draw [line width=1pt](-6.5,1) -- (-7.5,1);
\draw node at (-7,-1){$Q^\prime_1$};

\draw node at (3,-7){Figure 2};

\end{tikzpicture}
\end{center}

Now we return to the proof of the lemma. If $(\cup _{i=2}^{p}Q_i)\cap V_2=\emptyset$, then $L_n[\cup_{i=3}^{p}V_i]$ contains
$$Q(p-2,k)\subseteq\overline{K}_{k-1}\vee T(pk-2k,p-2)$$
 as a subgraph, we are done. By Claim 2, suppose that there exists a $j\in \{2,\ldots,p\}$ such that $V_2\cap Q_j \neq \emptyset$ and $V_2\cap (\cup _{i=2}^{p}Q_i\setminus Q_j) =\emptyset$. Since $|Q_1^\prime|=\ell\geq 2$, by Observation 1, we have $|V_2\cap Q_j| \leq k-2$. If $j=p$, then $L_n[V_3\cup \ldots \cup V_p]$ contains
 $$Q(p-2,k)\subseteq K_1\vee T(pk-2k,p)$$
  as a subgraph, we are done. Hence we have $j\in\{2,\ldots,p-1\}$. Without loss of generality, let $(\cup_{i=3}^{p}Q_{i})\cap V_2=\emptyset$ and $Q_2\cap V_{2}\neq \emptyset$. By $|V_2\cap Q_2| \leq k-2$, we have $|Q_2\cap( \cup_{i=3}^{p} V_i)|\geq 2$. By Claim 1, we have $e(L_n[Q_2\cap( \cup_{i=3}^{p} V_i)])=0$. Hence, without loss of generality, we have $|Q_2\cap V_3|\geq 2$ and $Q_2\cap( \cup_{i=4}^{p} V_i)=\emptyset$.    By  Claims 1, 2 and  Observation 1,  we can go on this procedure, and finally get   $|Q_{p-1}\cap V_p|\geq 2$ and $|Q_p\cap V_p|\leq k-2$. Since $Q_p\cap(\cup_{i=2}^{p-1} V_i)=\emptyset$, we have  $|Q_p\cap(\cup_{i=1}^{p} V_i)|= |Q_p\cap V_1|+|Q_p\cap  V_p|  \leq k-1$, a contradiction. Thus we finish the proof of lemma.\end{proof}

\section{Proof of the main theorems}

\noindent Denote by $T(Np,p;F)$ the graph obtained by embedding an $F$ in one partite set of $T(Np,p)$ and $T(Np,p;F_1,F_2)$ the graph obtained from $T(Np,p)$ by adding an $F_1$ into one partite set  and an $F_2$ into another partite set of $T(Np,p)$.\\

\noindent {\bf Proof of Theorem~\ref{main1}:} First, we present a useful proposition.
\begin{proposition}\label{proposition0} Let $\mathcal{M}_0(\mathcal{F})=\{M_2\}$, $\mathcal{M}_1(\mathcal{F})=\{M_{2k-2}\}$, $v=\max\{v(F):F\in \mathcal{F}\}$ and $k\geq 2$. Let $c_n$ be a coloring of $K_n$ and $L_n$ be a representing graph of $c_n$. If $L_n$ contains a copy of $T(2vp,p;M_{2k})$, then $K_n$ contains a rainbow copy of some $F\in \mathcal{F}$.
\end{proposition}
\begin{proof} Since $\mathcal{M}_0(\mathcal{F})=\{M_2\}$ and $\mathcal{M}_1(\mathcal{F})=\{M_{2k-2}\}$, there exist a graph $F\in \mathcal{F}$ and an edge $e$ such that $F-e \subseteq T(vp,p;M_{2k-2})$. Since $p(\mathcal{F}^{-})=p$ and $\mathcal{M}_0(\mathcal{F})=\{M_2\}$, we have $\chi(F)=p+2$. Therefore, $F\subseteq T(vp,p;M_{2k-2},M_2)$. Without loss of generality, let $B_1,\ldots,B_p$ be the $p$ partite vertex sets of $T(2vp,p;M_{2k})$  in $L_n$ and $L_n[B_1]$ contains an $M_{2k}$ as a subgraph. Let $\{x,y\}\subseteq B_2$ and $c_n(xy)$ be the color of $xy$. It is easy to see that after deleting the edge in $L_n$ which is colored by $c_n(xy)$ in $K_n$ and adding the edge $xy$ to $L_n$, the obtained graph contains $T(vp,p;M_{2k-2},M_2)$ as a subgraph. Hence $K_n$  contains a rainbow copy of some $F\in \mathcal{F}$.\end{proof}
First, we show that the coloring described in Theorem~\ref{main1} (\romannumeral2) is $\mathcal{F}$-free. If the representing graph $S_n=K_{k-2}\vee T(n-k+2,p;M_2)$ of this coloring contains an $F\in \mathcal{F}$ as a subgraph, then there are an $F\in \mathcal{F}$ and an edge $e\in F$ such that  $F-e \subseteq K_{k-2}\vee T(n-k+2,p)$. Hence $\mathcal{M}_1(\mathcal{F})$ contains a graph $H\neq M_{2k-2}$, a contradiction. Thus the coloring described in Theorem~\ref{main1} (\romannumeral2) is $\mathcal{F}$-free. For Theorem~\ref{main1} (\romannumeral1), if $\overline{K}_{k-2}\vee T(n-k+2,p;M_2)$ contains some $F\in \mathcal{F}$ as a subgraph, then there are an $F\in \mathcal{F}$ and an edge $e\in F$ such that  $F-e \subseteq \overline{K}_{k-2}\vee T(n-k+2,p)$. Hence $\mathcal{M}_0(\mathcal{F})$ contains a graph $H\neq M_{2k}$, a contradiction. Hence, by the definition of $q$, we have $q\geq1$. Thus the colorings described in Theorem~\ref{main1} (\romannumeral1) are $\mathcal{F}$-free. We will prove Theorem~\ref{main1} (\romannumeral1) (Theorem~\ref{main1} (\romannumeral2) resp.) by progressive induction. Suppose that $c_n$ is an extremal $\mathcal{F}$-free coloring of $K_n$ that uses at least $h^\prime(n,p,k-1)+q$ ($h(n,p,k-1)+1$ resp.) colors. It will be shown that, if $n$ is sufficiently large, then $c_n$ belongs to the coloring set described in the theorem. Let $L_n$ be a representing graph of $c_n$. Obviously, we have

\begin{equation}\label{1}
e(L_n)\geq h^\prime(n,p,k-1) +q
\end{equation}
\begin{equation}\label{1.1}
\left(e(L_n)\geq h(n,p,k-1) +1 \mbox{ resp.}\right).
\end{equation}
Hence
\[
\phi(c_n)=e(L_n)-(h^\prime(n,p,k-1) +q)
\]
\[
\left(\phi(c_n)=e(L_n)-(h(n,p,k-1) +1) \mbox{ resp.}\right)
\]
is a non-negative integer. The theorem will be proved by progressive induction, where $\mathfrak{U}_n$ is the set of extremal $\mathcal{F}$-free colorings of $K_n$. $B$ states that the coloring of $K_n$ belongs to the coloring set  described in the theorem, and $\phi(c_n)$ is a non-negative integer. According to the lemma of progressive induction, it is enough to show that if $c_n$ does not belong to the coloring set described in theorem, then there exists a $c_{n^\prime}$ with $n/2<n^{\prime}<n$ such that $\phi(c_{n^\prime})>\phi(c_n)$ provided $n$ is sufficiently large, where $c_{n^\prime}$ is an extremal $\mathcal{F}$-free coloring of $K_{n^\prime}$.
By Theorem~\ref{erdos-stone} and (\ref{1}) ((\ref{1.1}) resp.), there is an $n_1$, if $n>n_1$, then $L_n$ contains $T(n_2p,p)$ ($n_2$ is sufficiently large) as a subgraph. Any partite class of $T(n_2p,p)$ can not contain $M_{2k}$ as a subgraph, otherwise $L_n$ contains a rainbow copy of some $F\in\mathcal{F}$ (by Proposition~\ref{proposition0}, resp.), a contradiction. Hence there is an induced subgraph $T(n_3p,p)$ of $L_n$ with partite set $\widehat{B}_1,\ldots,\widehat{B}_p$, where $n_3\geq n_2-2(k-1)$. In fact, let $x_1y_1,x_2y_2,\ldots, x_{s_1}y_{s_1}$ be a maximal matching in one class, say $\widehat{B}_1$, of $T(n_2p,p)$, $\widetilde{B}_1=\widehat{B}_1-\{x_1,y_1,\ldots, x_{s_1},y_{s_1}\}$. Then there is no edge in $L_n[\widetilde{B}_1]$ and  there is an induced subgraph $T(n_3p,p)$ of $L_n$.

Let $\epsilon$ be a small constant satisfying
\begin{equation}\label{small}
\epsilon< \frac{1}{k+1}.
\end{equation}
Let $\widetilde{L}=L_n-T(n_3p,p)$. We partition $\widetilde{L}$ by the following produce.
If there is an $x_1\in \widetilde{L}$ which is joint to all the classes of $T(n_3p,p)=T_0$ by more than $\epsilon^2n_3$ vertices, then $T_0$ contains a $T_1=T(\epsilon^2n_3p,p)$ each vertex of which is joint to $x_1$. Generally, if there is an $x_i\in \widetilde{L}$ which is joint to at least $\epsilon^{2i}n_3$ vertices of each class of $T_{i-1}$, then there is a $T_i=T(\epsilon^{2i}n_3p,p) \subseteq T_{i-1}$ each vertices of which is joint to all the vertices $x_1,x_2,\ldots,x_i$. Thus we may define recursively a sequence of graphs. However, this process stops at last after the construction of $T_{k-2}$. Since if we could find a $T_{k-1}\subseteq L_n$, then $K_n$ contains a rainbow copy of some $F\in\mathcal{F}$, a contradiction. In fact, let $c_n(xy)$ be the color of $xy$ such that $x,y$ are in the same partite set of $T_{k-1}$, then there is at most one edge of $L_n[\overline{K}_{k-1}\vee T_{k-1}]$ which is colored by $c_n(xy)$ (note that $L_n[\overline{K}_{k-1}\vee T_{k-1}]$ is an induced subgraph of the representing graph $L_n$). Since $\epsilon^{2k}n_3$ is sufficiently large,  $L_n$ contains a copy of $T(2vp,p;M_{2k})$. Thus $K_n$ contains a rainbow copy of some $F\in\mathcal{F}$ (by Proposition~\ref{proposition0} resp.).

Now suppose the above progress ends at $T_\ell$, $0\leq \ell\leq k-2$. Let $x_1,x_2,\ldots,x_\ell$ be the vertices which are joint to all the vertices of $T_\ell$. Denote by $B_{\ell_1},B_{\ell_2},\ldots,B_{\ell_p}$ the classes of $T_\ell$, partition the remaining vertices into the following vertex sets: If $x$ is joint to less than $\epsilon^{2\ell+2}n_3$ vertices of $B_{\ell_i}$ and joint to every $B_{\ell_j\neq \ell_i}$ to more than $(1-\epsilon)\epsilon^{2\ell}n_3$, then $x\in C_{\ell_i}$. If $x$ is joint to less than $\epsilon^{2\ell+2}n_3$ vertices of $B_{\ell_i}$ and is joint to some of $B_{\ell_j\neq \ell_i}$ to less than $(1-\epsilon)\epsilon^{2\ell}n_3$, then $x\in D$. Obviously, this is a partition of $S_n-T_\ell-\{x_1,x_2,\ldots,x_\ell\}$. Since $\mathcal{M}_0(\mathcal{F})=M_{2k}$ ($\mathcal{M}_0(\mathcal{F})=M_{2}$, $\mathcal{M}_1(\mathcal{F})=M_{2k-2}$ resp.) and any vertex of $C_{\ell_i}$ is joint to less than $\epsilon^{2\ell+2}n_3$ vertices of $B_{\ell_i}$, there are $\epsilon^{2\ell}n_3(1-\epsilon^2k)$ vertices of $B_{\ell_i}$ which is not joint to any vertices of $C_{\ell_i}$. In fact, there are at most $k-1$ independent edges in $B_{\ell_i}\cup C_{\ell_i}$, otherwise, $L_n$ contains $T(2vp,p;M_{2k})$ as a subgraph, a contradiction (by Proposition~\ref{proposition0} resp.). Consider the edges joining $B_{\ell_i}$ and $C_{\ell_i}$ and select a maximal set of independent edges, says $x_1y_1,\ldots,x_qy_q$, $x_{i^{\prime}}\in B_{\ell_i}$, $y_{i^{\prime}}\in C_{\ell_i}$, $1\leq i^{\prime}\leq q\leq k-1$, among them, then the number of vertices of $B_{\ell_i}$ which are joint to at least one of $y_1,y_2,\ldots,y_q$ is less than $\epsilon^{2\ell+2}n_3q$, and the remaining vertices of $B_{\ell_i}$ is not joint to any vertices of $C_i$ by the maximality of $x_1y_1,\ldots,x_qy_q$. Hence we can move $\epsilon^{2\ell+2}n_3k$ vertices of $B_{\ell_i}$ to $C_{\ell_i}$, obtain $B_i$ and $C_i$ such that $B_i\subseteq B_{\ell_i}$, $C_{\ell_i}\subseteq C_i$ and there is no edge between $B_i$ and $C_i$. Let $\ell^\prime=(1-\epsilon^2k)\epsilon^{2\ell}n_3$, we conclude that $T^{\prime}_\ell=T(\ell^\prime p,p)$ with classes $B_1,\ldots,B_p$ is an induced subgraph of $L_n$ satisfying the following conditions:\\
Let $\widehat{L}_{n-\ell^\prime p}=L_n-T^{\prime}_\ell$. The vertices of $\widehat{L}_{n-\ell^\prime p}$ can be partitioned into $p+2$ classes $C_1,\ldots,C_p,D,$ $E$ such that
\begin{itemize}
\item Every $x\in E$ is joint to every vertex of $T^{\prime}_\ell$ and $|E|=\ell$.
\item If $x\in C_i $ then $x$ is joint to at least $(1-\epsilon-\epsilon^2k)\epsilon^{2\ell}n_3$ vertices of $B_{j\neq i}$ and is joint to no vertex of $B_i$.
\item If $x\in D$ then there are two different classes of $T^{\prime}_\ell$: $B_{i(x)}$ and $B_{j(x)}$ such that $x$ is joint to less than $(1-\epsilon)\epsilon^{2\ell}n_3$ vertices of $B_{i(x)}$  and less than $\epsilon^{2\ell+2}n_3$ vertices of $B_{j(x)}$.
\end{itemize}
Denote by $e_S$ the number of the edges joining $\widehat{L}_{n-\ell^\prime p}$ and $T^\prime_\ell$. Clearly
\begin{equation}\label{2}
e(L_n)=e(T^\prime_\ell)+e_S+e(\widehat{L}_{n-\ell^\prime p}).
\end{equation}
Let $L^\prime_n$ be a representing graph of $c^\prime_n$, where $c^\prime_n$ is an $\mathcal{F}$-free coloring described in the theorem. Similarly, select an induced subgraph $T^{\prime}_\ell$ of $L^\prime_n$. Let $$L_{n-\ell^\prime p}^\prime=L^\prime_n-T^{\prime}_\ell$$ and $e_T$ denote the number of edges of $L^\prime_n$ joining $T^{\prime}_\ell$ with $L_{n-\ell^\prime p}^\prime$. Then we have
\begin{equation}\label{3}
e(L^\prime_n)=e(T^\prime_\ell)+e_T+e(L_{n-\ell^\prime p}^\prime).
\end{equation}
Since $\widehat{L}_{n-\ell^\prime p}$ does not contain any $F\in \mathcal{F}$ as a subgraph, we have $e(\widehat{L}_{n-\ell^\prime p})\leq e(L_{n-\ell^\prime p})$, where $L_{n-\ell^\prime p}$ is a representing graph of an extremal $\mathcal{F}$-free coloring graph on $n-\ell^\prime p$ vertices. By (\ref{2}) and (\ref{3}), we have
\begin{eqnarray*}
\phi(c_n)&=&e(L_n)-e(L^\prime_n)
\\&=&e(T^\prime_\ell)-e(T^\prime_\ell)+(e_S-e_T)+e(\widehat{L}_{n-\ell^\prime p})-e(L^\prime_{n-\ell^\prime p})
\\&\leq&(e_S-e_T)+e(L_{n-\ell^\prime p})-e(L^\prime_{n-\ell^\prime p})
\\&=&(e_S-e_T)+\phi(c_{n-\ell^\prime p}),
\end{eqnarray*}
where $c_{n-\ell^\prime p}$ is an extremal $\mathcal{F}$-free coloring of $K_n$. If $e_S-e_T<0$, then $\phi(c_n)<\phi(c_{n-\ell^\prime p})$, we are done. Hence we may assume $e_S-e_T\geq 0$. Since $\epsilon$ is a small constant, by (\ref{small}) we have
\begin{eqnarray*}
e_S-e_T&\leq&\ell\cdot \ell^\prime p+(n-\ell-\ell^\prime p-|D|)\cdot \ell^\prime (p-1)\\
&&+|D|\cdot [\ell^\prime (p-2)+(1-\epsilon)\epsilon^{2\ell}n_3+\epsilon^{2\ell+2}n_3]\\
&&-[(k-2)\cdot \ell^\prime p+(n-k+2-\ell^\prime p)\cdot \ell^\prime (p-1)]\\
&\leq&  \ell\cdot \ell^\prime p +(n-\ell-\ell^\prime p)\cdot \ell^\prime (p-1) \\
&&-[(k-2)\cdot \ell^\prime p+(n-k+2-\ell^\prime p)\cdot \ell^\prime (p-1)]\\
&\leq &  0,
\end{eqnarray*}
where equality holds if and only if $|D|=0$, $\ell=k-2$ and each vertex of $C_i$ is joint to each vertex of $B_{j\neq i}$ in $L_n$ for $i=1,\ldots,p$. Moreover, there is only one color in $K_n[B_i\cup C_i]$ for $i=1,\ldots,p$ (there is only one color in $\cup_{i=1}^{p}K_n[B_i\cup C_i]$ resp.). In fact, if there are two colors in $K_n[B_i\cup C_i]$ for some $i\in\{1,\ldots,p\}$ (there are two colors in $K_n[B_i\cup C_i\cup B_j\cup C_j]$ for some $i\neq j$ resp.), since we can choose any edge with the same color for the representing graph, there is a representing  graph $L^*_n$ which contains a copy of $T(2vp,p;M_{2k})$ (a copy of $T(2vp,p;M_{2k})$ or $T(2vp,p;M_{2k-2},M_2)$ resp.). Thus $K_n$ contains a rainbow copy of some $F\in\mathcal{F}$ (by Proposition~\ref{proposition0} resp.).

Now we prove Theorem~\ref{main1} (\romannumeral2). Since $T(n-k+2,p)$ has more edges than any other $p$-partite graph, by (\ref{1.1}), we have that each vertex in $E$ is joint any other vertex in $L_n$, $\lfloor(n-k+2)/p \rfloor\leq |B_i\cup C_i|\leq\lceil(n-k+2)/p \rceil$ and each vertex in $B_i\cup C_i$ is joint to each vertex in $B_{j\neq i}\cup C_{j\neq i}$ in $L_n$ for $i=1,\ldots,p$. Hence $c_n$ belongs to the coloring set described in Theorem~\ref{main1} (\romannumeral2).

For Theorem~\ref{main1} (\romannumeral1), since $T(n-k+2,p)$ has more edges than any other $p$-partite graph, by (\ref{1}), there are at least $q$ edges in $\cup_{i=1}^pL_n[B_i\cup C_i]\cup L_n[E]$. Moreover, by the definition of $q$, there are at most $q$ colors in $\cup_{i=1}^pK_n[B_i\cup C_i]\cup K_n[E]$. Otherwise, since every $x\in E$ is joint to every vertex of $T^{\prime}_{k-2}$ and there is only one color in $K_n[B_i\cup C_i]$ for $i=1,\ldots,p$,   $L_n$ contains some $F\in \mathcal{F}$ as a subgraph. Hence, it follows from (\ref{1}) and $T(n-k+2,p)$ has more edges than any other $p$-partite graph that, each vertex in $E$ is joint to each vertex in $\cup_{i=1}^{p}(B_i\cup C_i)$, $\lfloor(n-k+2)/p \rfloor\leq |B_i\cup C_i|\leq\lceil(n-k+2)/p \rceil$ and each vertex in $B_i\cup C_i$ is joint to each vertex in $B_{j\neq i}\cup C_{j\neq i}$ in $L_n$ for $i=1,\ldots,p$, the result follows.\rule{3mm}{3mm}\par\medskip



\noindent{\bf Proof of Theorem~\ref{main2}:}
We only prove Theorem~\ref{main2} (\romannumeral2), since the proof of Theorem~\ref{main2} (\romannumeral1) is essentially the same as the proof of Theorem~\ref{main2} (\romannumeral2). As the proof of Theorem~\ref{main1}, we present a useful proposition.
\begin{proposition}\label{proposition1}
Let $v=\max\{v(F):F\in \mathcal{F}\}$ and $k\geq3$. Let $c_n$ be a coloring of $K_n$ and  $L_n$ be a representing graph of $c_n$.\\
(\romannumeral1) If $L_n$ contains a copy of $T(vp,p;S_{k+1},M_2)$, then $K_n$ contains a rainbow copy of some $F\in \mathcal{F}$.\\
(\romannumeral2) If $L_n$ contains a copy of $T(2vp,p;H)$, then $K_n$ contains a rainbow copy of some $F\in \mathcal{F}$, where  $H\notin\{K_{2,k},S_{2k+1}\}$ is a graph contains  two edge-disjoint copies of $S_{k+1}$.\\
(\romannumeral3) If $L_n$ contains a copy of $Q(p,m)$, then $K_n$ contains a rainbow copy of some $F\in \mathcal{F}$, where $m=m(v,k)$ is a large constant depending on $v$ and $k$.
\end{proposition}
\begin{proof}
(\romannumeral1) Since $\mathcal{M}_0(\mathcal{F})=\{M_2\}$ and $\mathcal{M}_1(\mathcal{F})=\{S_{k+1}\}$, there exist a graph $F\in \mathcal{F}$ and an edge $e$ such that $F-e \subseteq T(vp,p;S_{k+1})$. Moreover, since $\mathcal{M}_0(\mathcal{F})=\{M_2\}$ and $p(\mathcal{F}^{-})=p$, we have $p(\mathcal{F})=p+1$. Suppose that $F\nsubseteq T(vp,p;S_{k+1},M_2)$. Then $F\subseteq T(vp,p;H_1)$, where $H_1$ is obtained from $S_{k+1}$ by adding an edge. If $H_1$ does not contain a triangle, then we have $p(\mathcal{F})=p$, a contradiction. If $H_1$ contains a triangle, by $k\geq 3$, we have $\mathcal{M}_1(\mathcal{F})\neq\{S_{k+1}\}$ ($\mathcal{M}_1(\mathcal{F})$ contains a tree on $k+1$ vertices with $k$ edges which is not $S_{k+1}$), a contradiction. The result follows.

(\romannumeral2) Without loss of generality, let $B_1,\ldots,B_p$ be the $p$ partite vertex sets of $T(2vp,p;H)$ in $L_n$ where $L_n[B_1]$ contains $H$ as a subgraph.  Let $\{x,y\}\subseteq B_2$ and $c_n(xy)$ be the color of $xy$. Since the graph $H$ satisfies that after deleting any edge or any vertex of it the resulting graph contains $S_{k+1}$ as a subgraph. we obtain that   after deleting the edge in $L_n$ which is colored by $c_n(xy)$ in $K_n$  and adding the edge $xy$ to $L_n$, the obtained graph contains $T(vp,p;S_{k+1},M_2)$ as a subgraph. Hence, $K_n$ contains a rainbow copy of some $F\in \mathcal{F}$.

(\romannumeral3) Let $Q_1,\ldots,Q_p$ be the $p$ partite vertex sets of $Q(p,m)$ and $x$ be the unique vertex in $Q(p,m)$ which is joint to all other vertices. Let $Q_i=\{x_{i,1},\ldots,x_{i,m}\}$ for $i=1,\ldots,p$.\\

{\bf Claim.} Let  $s\neq t$ and  $\{s,t\} \subset \{1,\ldots,m\}$. We have $c_n(x_{i,s}x_{i,t})\in \{c_n(xx_{i,s}),c_n(xx_{i,t})\}$ for $i=1,\ldots,p$.\\

\begin{proof} Otherwise, since $m$ is large, $K_n$ contains a rainbow copy of $T(vp,p;S_{k+1},M_2)$. The result follows from (\romannumeral1).\end{proof}

Moreover, there are at least $m-k$ edges of  $\{x_{i,j}x_{i,1},\ldots,x_{i,j}x_{i,m}\}$ which are colored by $c_n(xx_{i,j})$ for $i=1,\ldots,p$ and $j=1,\ldots,m$. Otherwise, by the claim, there are a rainbow $S_{k+1}$ in $K_n[Q_i]$ which is colored by the colors in $\{c_n(xx_{i,1}),\ldots,c_n(xx_{i,m})\}$ and an edge in $K_n[V_q]$ which is colored by a color in $\{c_n(xx_{q,1}),\ldots,c_n(xx_{q,m})\}$ for $i\neq q$. Hence $K_n$ contains a rainbow copy of $T(vp,p;S_{k+1},M_2)$, we are done. Since $m$ is large, without lose of generality, there is a rainbow $S_{k+1}$ in $K_n[Q_1]$ which is colored by $c_n(xx_{1,1}),\ldots,c_n(xx_{1,k})$ (this follows  from the fact that there are at least $m-k$ edges of  $\{x_{1,\ell}x_{1,1},\ldots,x_{1,\ell}x_{1,m}\}$ which are colored by $c_n(xx_{1,\ell})$ for $\ell=1,\ldots,k$) and an edge in $K_n[Q_{2}]$ which is colored by $c_n(xx_{2,1})$. Thus, $K_n$ contains a rainbow copy of $T(vp,p;S_{k+1},M_2)$ and we finish the proof of the proposition.\end{proof}

\begin{proof} Let $v=\max\{v(F):F\in \mathcal{F}\}$, $\epsilon$ be a small constant only depending on $v$ and $k$ and $N$ be a large even constant depending on $v$ and $k$. We will prove this theorem by progressive induction. Suppose that $c_n$ is an extremal $\mathcal{F}$-free coloring  of $K_n$. It will be shown that, if $n$ is sufficiently large, then $c_n$ belongs to the coloring set described in the theorem. Let $L_n$ be a representing graph of $c_n$. By Lemma~\ref{lemma3}, we have
\begin{equation}\label{4}
e(L_n)\geq t(n,p)+\left\lfloor\frac{(k-2)\lceil\frac{n}{p}\rceil}{2}\right\rfloor+\ldots+\left\lfloor\frac{(k-2)\lfloor\frac{n}{p}\rfloor}{2}\right\rfloor+p
\end{equation}
Hence
\begin{equation}\label{6}
\phi(c_n)=e(L_n)-t(n,p)-\left\lfloor\frac{(k-2)\lceil\frac{n}{p}\rceil}{2}\right\rfloor+\ldots+\left\lfloor\frac{(k-2)\lfloor\frac{n}{p}\rfloor}{2}\right\rfloor-p
\end{equation}
 is a non-negative integer. The theorem will be proved by progressive induction, where $\mathfrak{U}_n$ is the set of extremal $\mathcal{F}$-free colorings of $K_n$. $B$ states that the coloring of $K_n$ belongs to the coloring set described in the theorem, and $\phi(c_n)$ is a non-negative integer. According to the lemma of progressive induction, it is enough to show that if $c_n$ does not belong to the coloring set described in theorem, then there exists an $n^{\prime}$ with $n/2<n^{\prime}<n$ such that $\phi(c_{n^{\prime}})>\phi(c_n)$ provided $n$ is sufficiently large, where $c_{n^\prime}$ is an extremal $\mathcal{F}$-free coloring  of $K_{n^\prime}$. By Theorem~\ref{erdos-stone} and (\ref{4}), there exists an $n_1$ such that if $n>n_1$, then  $L_n$ contains $T(Np,p)$ as a subgraph.

Suppose that $L_n$ does not contain $T(Np/3,p;S_{k+1})$ as a subgraph. Let $T^\prime(Np,p)$ be the subgraph of $L_n$ induced by $V(T(Np,p))$ and $\widetilde{L}_{n-Np}=L_n-T^\prime(Np,p)$. We can partition the vertices of $\widetilde{L}_{n-Np}$ into $C_1,\ldots,C_p,D$ such that if $x\in C_i$ then $x$ is joint to less than $k$ vertices of $B_i$ and more than $(1-\epsilon)N$ vertices of $B_{j\neq i}$, if $x\in D$ then $x$ is joint to at most $(1-\epsilon)N$ vertices of each of two of $B_1,\ldots,B_p$. Furthermore, since $L_n$ does not contain $T(Np/3,p;S_{k+1})$ as a subgraph, if $x\in B_i\cup C_i$, then $x$ is joint to less than $k$ vertices of $B_i\cup C_i$ for $i\in\{1,\ldots,p\}$, and if $x\in D$, then $x$ is joint to less than $(p-1-\epsilon)N+k-1$ vertices of $T^\prime(Np,p)$ (if $x$ is joint to more than $k$ vertices of $B_i$ for some $i\in\{1,\ldots,p\}$, since $L_n$ does not contain $T(Np/3,p;S_{k+1})$ as a subgraph, then $x$ is joint to at most $N/3$ vertices of each of two of $B_1,\ldots,B_p$).
Denote by $e_L$ the number of edges joining $\widetilde{L}_{n-Np}$ and $T^\prime(Np,p)$. We have
\begin{equation}\label{8}
e(L_n)=e(T^\prime_p(Np))+e_L+e(\widetilde{L}_{n-Np}).
\end{equation}
Let $c^\prime_n$ be a coloring of $K_n$ which belongs to the coloring set described in the theorem such that in each partite set there is a rainbow component on $N$ vertices. Since $N$ is a large constant, by Proposition~\ref{k-1 regular}, this is possible. Let $L^\prime_n$ be a representing graph of $c^\prime_n$.  Hence we can choose a subgraph $T^\ast(Np,p)=T(Np,p;\mathcal{U}_{Np,k-1}^{\prime})$  of $L^\prime_n$, and
\begin{equation}\label{9}
e(L^\prime_n)=e(T^\ast(Np,p))+e_{L^\prime}+e(L^\prime_{n-Np}),
\end{equation}
where $L^\prime_{n-Np}$ is a representing graph of $c^\prime_{n-Np}$ and $c^\prime_{n-Np}$ is a coloring of $K_{n-Np}$ belongs to the coloring set described in the theorem. Obviously, $e_{L^\prime}=(n-Np)N(p-1)$.

Since $T^\prime(Np,p)$ does not contain $T(Np/3,p;S_{k+1})$ as a subgraph, we have $e(T^\prime(Np,p))\leq e(T^\ast(Np,p))+N_1$, where $N_1=(k-1)pN/2$. Let $L_{n-Np}$ be a representing graph of an extremal $\mathcal{F}$-free coloring on $K_{n-Np}$. By (\ref{8}) and (\ref{9}), we have
\begin{eqnarray*}
\phi(c_n)&=&e(L_n)-e(L^\prime_n)
\\&=&e(T^\prime(Np,p))-e(T^\ast(Np,p))+(e_{L}-e_{L^\prime})+e(\widetilde{L}_{n-Np})-e(L^\prime_{n-Np})
\\&\leq&(e_{L}-e_{L^\prime})+e(L_{n-Np})-e(L^\prime_{n-Np})+N_1
\\&=&(e_{L}-e_{L^\prime})+\phi(c_{n-Np})+N_1,
\end{eqnarray*}
where $c_{n-Np}$ an extremal $\mathcal{F}$-free coloring  of $K_n$.
Thus
\begin{equation}\label{10}
\phi(c_n)\leq(e_{L}-e_{L^\prime})+\phi(c_{n-Np})+N_1.
\end{equation}
Let $c_{n-1}$ be an extremal $\mathcal{F}$-free coloring  of $K_{n-1}$. It will be proved that if $n$ is large enough, then
\begin{itemize}
\item (a) either $\phi(c_n)<\phi(c_{n-Np})$,
\item (b) or $\phi(c_n)<\phi(c_{n-1})$,
\item (c) or $c_n$ belongs to the coloring set described in the theorem.
\end{itemize}
This will complete our progressive induction.

If there is a vertex $x\in L_n$ with $d_{L_n}(x)<\lfloor n/p\rfloor(p-1)$, then $\phi(c_n)<\phi(c_{n-1})$. In fact, let $L_{n-1}^\ast=L_n-\{x\}$. It does not contain any $F\in \mathcal{F}$ as a subgraph. Thus $e(L_n)-d_{L_n}(x)=e(L_{n-1}^\ast)\leq e(L_{n-1})$ and from this we have
$e(L_n)-e(L_{n-1})\leq d_{L_n}(x)< \lfloor n/p\rfloor(p-1)$. Since $e(L^\prime_n)-e(L^\prime_{n-1})\geq\lfloor n/p\rfloor(p-1)$, we have $\phi(c_n)=e(L_n)-e(L^\prime_n)< e(L_{n-1})-e(L^\prime_{n-1})=\phi(c_{n-1})$.

Suppose now that neither (a) nor (b) holds. Then for each $x\in L_n$, we have $d_{L_n}(x)\geq \lfloor n/p\rfloor(p-1)$  and $\phi(c_n)\geq \phi(c_{n-Np})$.
From (\ref{10}) we have $0\leq\phi(c_n)-\phi(c_{n-Np})\leq e_L-e_{L^\prime}+N_1$.\\

{\bf Claim 1.} There is a constant $N_3$ such that $|D|\leq N_3$.\\

\begin{proof} First recall that $B_i\cup C_i$ does not contain such a vertex which is joint to $k$ other vertices of it. Thus the number of edges joining $B_i$ and $C_i$ is less than $Nk$ and
\begin{equation}\label{11}
e_{L}\leq (n-Np)(p-1)N+Nkp-|D|N_2=e_{L^\prime}+Nkp-|D|N_2,
\end{equation}
since if $x\in D$, then $x$ is joint to less than $(p-1-\epsilon)N+k-1\leq (p-1)N-N_2$ vertices of $T^\prime(Np,p)$. By (\ref{11}) we have
 $|D|\leq (e_{L^\prime}-e_{L}+Nkp)/N_2\leq(N_1+Nkp)/N_2= N_3$.\end{proof}

{\bf Claim 2.} A vertex belonging to $B_i\cup C_i$ is joint to at most to $k-1$ other vertices of $B_i \cup C_i$.\\

\begin{proof} This claim was already proved.\end{proof}

{\bf Claim 3.} $|B_i\cup C_i|=n/p+O(\sqrt{n}).$\\

\begin{proof} In order to show this, omit the edges joining two vertices of the same $B_i\cup C_i$ ($i=1,\ldots,p$) and the edges of $D$. Thus
there remains an $H_{n-|D|}$ which is $p$-chromatic and has $t(n,p)-O(n)$ edges. Applying Lemma~\ref{lemma1} to $H_{n-|D|}$ we obtain the required result. Thus there is a constant $N_4$ such that
$$\left||B_i\cup C_i|-\frac{n}{p}\right|\leq N_4 \sqrt{n}.$$
The result follows.\end{proof}

{\bf Claim 4.} There is a constant $N_5$ such that every $x\in B_i\cup C_i$ is joint to all the vertices of $L_n -(B_i\cup C_i)$ except less than $N_5 \sqrt{n}$ vertices.\\

\begin{proof} This follows immediately from the fact that $x$ is not joint to $n/p- N_4 \sqrt{n} - k$ vertices of $ B_i\cup C_i$ but $\lfloor n/p\rfloor(p-1)\leq d_{L_n}(x)< n $.\end{proof}

Let $D_i$ be the class of those vertices in $D$, which are joint to $B_i\cup C_i$ by less than $k$ edges for $i=1,\ldots,p$.\\

{\bf Claim 5.} $D$ is the disjoint union of $D_1,\ldots,D_p$.\\

\begin{proof} In fact, if $x\in D$, then there is an $i(x)$ such that $x$ is joint to at least $(1/3)(n/p)$ vertices of $B_{j\neq i(x)}\cup C_{j\neq i(x)}$. Otherwise $d_{L_n}(x)<(p-2)\lfloor n/p\rfloor+O(\sqrt{n})+(2/3)(n/p)<\lfloor n/p\rfloor(p-1)$, a contradiction. Furthermore, $x\in D$ is joint to less than $k$ vertices of $B_{i(x)}\cup C_{i(x)}$, otherwise, $L_n$ contains $T(Np/3,p;S_{k+1})$ as a subgraph (without loss of generality, let $i(x)=1$, we select $k$ vertices of $B_1\cup C_1$ which are joint to $x$ and $N/3-k-1$ other vertices of $B_1\cup C_1$. Then select $N/3$ vertices in $B_2\cup C_2$ which are joint to $ x$ and to the $N-1$ vertices considered in $B_1\cup C_1$. Let us continue this selection and lastly select $N/3$ vertices of $B_p\cup C_p$ which are joint to all the $(p-1)N/3$ vertices from $B_1\cup C_1,\ldots,B_{p-1}\cup C_{p-1}$. This is possible, since each vertex selected from is joint to at least $n/p-O(\sqrt{n})$ of $B_{i+1}\cup C_{i+1}$ and $x$ is joint to at least $(1/3)(n/p)$ vertices of $B_{i+1}\cup C_{i+1}$ for $i=1,\ldots,p-1$.). The result follows.\end{proof}

Now we may suppose that $L_n$ contains $T(Np/3,p;S_{k+1})$ as a subgraph with vertex partition $B_1\cup\ldots\cup B_p$. We can partition the vertices of $\widetilde{L}_{n-Np/3}=L_n-V(T(Np/3,p;S_{k+1}))$ into $C_1,\ldots,C_p,D$ such that if $x\in C_i$ then $x$ is joint to more than $(1-\epsilon)N/3$ vertices of $B_{j\neq i}$, if $x\in D$ then $x$ is joint to at most $(1-\epsilon)N/3$ vertices of each of two of $B_1,\ldots,B_p$. Furthermore,  there exists  one vertex in  $B_i\cup C_i$ such that after deleting it any vertex in $B_i\cup C_i$ is joint to less than $k+1$ vertices of $B_i\cup C_i$ for $i=1,\ldots,p$. Otherwise $L_n[B_i\cup C_i]$ contains a graph $H\notin \{S_{2k+1},K_{2,k}\}$ which contains two edge-disjoint copies of $S_{k+1}$. Since $x\in  B_i\cup C_i$ is joint to more than $(1-\epsilon)N/3$ vertices of $B_{j\neq i}$, we have $L_n$ contains $T(Np/3,p;H)$ as a subgraph. Hence, by Proposition~\ref{proposition1} (\romannumeral2), $L_n$ contains some $F\in \mathcal{F} $ as a subgraph, a contradiction. If  $x\in D$, then $x$ is joint to less than $m$ vertices of some $B_i$ for $i\in\{1,\ldots,p\}$. Otherwise $L_n$ contains $Q(p,m)$ as a subgraph.  Hence by Proposition~\ref{proposition1} (\romannumeral3), $L_n$ contains some $F\in \mathcal{F} $ as a subgraph, a contradiction. Thus $x$ is joint to less than $(p-1-\epsilon)N/3+m$ vertices of $T(Np/3,p;S_{k+1})$. Let $B^\prime_i\subseteq B_i$ and $C^\prime_i\supseteq C_i$  such that the unique vertex of $B_i\cup C_i$ which is joint to more than $k$ other vertices of $B_i\cup C_i$ lies in $C^\prime_i$ (if there is no such vertex, then move any vertex of $B_i$ to $C_i$).

Since the number of edges between $B^\prime_i$ and $C^\prime_i$ is at most $kN/3$ and $x\in D$ is joint to less than $(p-1-\epsilon)N/3+m$ vertices of $T(Np/3,p;S_{k+1})$, following from the previous proof, we get Claims 1, 3. For any $x\in B_i\cup C_i$, we have $d_{L_n[B_i\cup C_i]}(x)\leq m$. Otherwise, $L_n$ contains  $Q(p,m)$ as a subgraph, hence by Proposition~\ref{proposition1} (\romannumeral3), $L_n$ contains some $F\in \mathcal{F}$ as a subgraph, a contradiction. This proves Claim 4. Furthermore, if $\Delta(L_n[B_i\cup C_i])\geq k+1$, then there is a vertex $x_i$ such that $\Delta(L_n[B_i\cup C_i\setminus \{x_i\}])\leq k-1$. Otherwise, $L_n[B_i\cup C_i]$ contains an $H\notin\{K_{2,k},S_{2k+1}\}$ which contains  two edge-disjoint copies of $S_{k+1}$. Hence, by Claim 4, $L_n$ contains $T(2vp,p;H)$ as a subgraph, and by Proposition~\ref{proposition1} (\romannumeral2), $L_n$ contains some $F\in \mathcal{F} $ as a subgraph, a contradiction. If $\Delta(L_n[B_i\cup C_i])=k$, then after deleting $k$ vertices of $B_i\cup C_i$ the resulting graph is a graph with maximum degree less than $k$. Otherwise $L_n[B_i\cup C_i]$ contains an $H\notin\{K_{2,k},S_{2k+1}\}$ which contains  two edge-disjoint copies of $S_{k+1}$, a contradiction by Proposition~\ref{proposition1} (\romannumeral2). Thus
\begin{equation}\label{3.3}
e(L_n[B_i\cup C_i])\leq ((k-1)/2)(n/p)+O(\sqrt{n}).
\end{equation}

Let $D_i$ be the class of vertices in $D$ which are joint to $B_i\cup C_i$ by less than $m$ edges. Here we can also get Claim 5, that is $D$ is the disjoint union of $D_1,\ldots,D_p$. In fact, as in the previously proof of Claim 5, if $x\in D$, then there is an $i(x)$ such that $x$ is joint to at least $(1/3)(n/p)$ vertices of $B_{j\neq i(x)}\cup C_{j\neq i(x)}$. Furthermore, $x\in D$ is joint to less than $m$ vertices of $B_{i(x)}\cup C_{i(x)}$ for some $i(x)\in\{1,\ldots,p\}$. Otherwise, $L_n$ contains $Q(p,m)$ as a subgraph (the proof is the same as the proof of Claim 5 in the parenthesis), by Proposition~\ref{proposition1} (\romannumeral3), $L_n$ contains some $F\in \mathcal{F}$ as a subgraph, a contradiction.\\

Let  $E_i=B_{i}\cup C_{i}\cup D_{i}$. We have proved that $||E_i|-n/p|\leq N_6 \sqrt{n}$ and each vertex in $E_i$ is joint to $n/p-O(\sqrt{n})$ vertices of $E_{j\neq i}$ in $L_n$ for $i=1\ldots,p$, where $N_6$ is a large constant.\\



{\bf Case 1.} There is an $S_{k+1}$ in $L_n[E_i]$ for some $i\in \{1,\ldots,p\}$.\\

Choose any $2v-k-1$ vertices of $E_i$. Let $X$ be the vertex set which contains those $2v-k-1$ vertices and the vertices of $S_{k+1}$ in $E_i$. Then the number of common neighbours of $X$ in $E_{j\neq i}$ is $n/p-O(\sqrt{n})$. By Proposition~\ref{proposition1} (\romannumeral1), there is no edge (edges in $L_n$) in the common neighbours of $X$ in $E_{j\neq i}$. Thus, by (\ref{3.3}), we have $e(L_n[E_{i}])\leq((k-1)/2)(n/p)+O(\sqrt{n})$ and $e(L_n[E_{j\neq i}])=O(\sqrt{n})$.
Hence
$$e(L_n)\leq t(n,p)+((k-1)/2)(n/p)+O(\sqrt{n}),$$ a contradiction to (\ref{4}) for $k\geq 4$ or $p\geq 3$.

For $k=3$ and $p=2$, by (\ref{4}) we have $e(L_n[E_i])=n/2+O(\sqrt{n})$ and $e(L_n[E_{3-i}])= O(\sqrt{n})$. Moreover, there is only one color in $K_n[E_{3-i}]$. In fact, suppose that there are at least two colors in $K_n[E_{3-i}]$. Choose two independent edges with different colors, say $c_n(e_1)$ and $c_n(e_2)$, in $K_n[E_{3-i}]$  for $L_n$ ($L_n$ is obtained by changing at most two edges). Let $Y$ be the vertex set which contains the vertices incident with this two edges and $4v-4$ arbitrary vertices of  $E_{3-i}$. Then the number of common neighbours of $Y$  in $E_{3-i}$, denoted by $N(Y)$, is $n/2+O(\sqrt{n})$ and $e(L_n[N(Y)])=n/2+O(\sqrt{n})$. By Proposition~\ref{proposition1} (\romannumeral1), we  have $\Delta(L_n[N(Y)])\leq 2$, otherwise $L_n$ contains a copy of some $F\in \mathcal{F}$. Thus, there are  $n/2+O(\sqrt{n})$ vertices of $N(Y)$ in $L_n[N(Y)]$ with degree two. Choose four of them, say $z_1,z_2,z_3,z_4$, such that the vertices in $\cup_{i=1}^4N_{L_n[N(Y)]}[z_i]$ are distinct vertices (Let $N_G[x]=N_G(x)\cup \{x\}$). We have $c_n(z_1z_2)\neq c_n(z_3z_4)$ and $c_n(z_1z_2)$, $c_n(z_3z_4)$ belong to the colors which incident with one vertex of $Y$  and one vertex of $N(Y)$. Otherwise, $L_n$ contains $T(2v,2;S_4,M_2)$ as a subgraph and by Proposition~\ref{proposition1} (\romannumeral1), $L_n$ contains some $F\in \mathcal{F}$ as a subgraph, a contradiction. Hence, if we choose $z_1z_2$ and $z_3z_4$ for $L_n$, then $L_n[N(Y)]$ contains an $H\notin\{K_{2,3},S_{7}\}$ which contains  two edge-disjoint copies of $S_{4}$. Since each vertex in $E_i$ is joint to $n/2-O(\sqrt{n})$ vertices of $E_{3-i}$ in $L_n$, $L_n$ contains $T(4v,2;H)$ as a subgraph. Thus by Proposition~\ref{proposition1} (\romannumeral2), $L_n$ contains some $F\in \mathcal{F} $ as a subgraph, a contradiction. Now, choose any edge in $K_n[E_i]$ for $L_n$, we have $\Delta(L_n[E_{3-i}])=2$, otherwise $L_n$ contains  $T(2v,2;S_4,M_2)$ as a subgraph, a contradiction. The result follows  by (\ref{4}) and an easy calculation ($T(n,p)$ has more edges then any other $p$-partite graph).\\

{\bf Case 2.} There is no $S_{k+1}$ in $L_n[E_i]$ for $i=1,\ldots,p$.\\

There exists an edge in $L_n[E_i]$ such that after deleting it, for any $x\in E_i$, $x$ is joint to less than $k-1$ vertices of $E_i$ for  $i=1,\ldots,p$. Otherwise, we have $\Delta(L_n[E_i])=k-1$ and $\Delta(L_n[E_i]-\{z\})=k-1$, where $z$ is a vertex of $L_n[E_i]$ with degree $k-1$. Hence $L_n[E_i]$ contains two edge-disjoint $S_k$. Let $x$, $y$ be the vertices of each of the two $S_k$ with degree $k-1$, $c_n(xy)$ be the color of $xy$. Hence there is an $S_{k+1}$ in $L_n[E_i]$ (the graph $L_n$ is obtained by changing one edge), and we are done by Case 1. Thus, after deleting $p$ suitable edges of $L_n$, we can partition the obtained graph $\widetilde{L}_n$ into $E^\prime_1,\ldots,E^\prime_p$ such that $\Delta(\widetilde{L}_n[E^\prime_i])\leq k-2$ for $i=1,\ldots,p$. The result follows  by (\ref{4}) and an easy calculation ($T(n,p)$ has more edges than any other $p$-partite graph).\end{proof}

\remark Our method can also be applied to the case when $k\leq2$, since there is no new idea and the extremal colorings are rather complicated, we skip it. We only point out that, for $k=2$, Proposition~\ref{proposition1} (\romannumeral1) should be changed as following. If $L_n$ contain a copy of $T(vp,p;S_{3},M_2)$ and a copy of $T(vp,p;K_3)$, then $K_n$ contains a rainbow copy of some $F\in \mathcal{F}$.

\section{Conclusion}
If $\mathcal{F}$ only contains bipartite graphs, then $\mathcal{F}$ is determined by its decomposition family sequence. One may think that if $\mathcal{F}$ only contains the graphs with the same chromatic number, then its decomposition family sequence only contains bipartite graphs. The following graph shows that this is not true. Let $H$ be the graph obtained by a $C_6=x_1\ldots x_6$ and adding the edges $x_1x_3$, $x_3x_5$ and $x_5x_1$. Then $\mathcal{M}_0(\{H\})=\{K_3,P_4\}$.

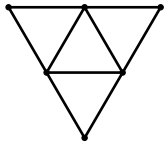
\begin{center}
\begin{tikzpicture}[scale = 0.5]
\filldraw[fill=black] (0,0) circle(2pt);
\filldraw[fill=black] (-1,1.732) circle(2pt);
\filldraw[fill=black] (1,1.732) circle(2pt);
\filldraw[fill=black] (0,3.464) circle(2pt);
\filldraw[fill=black] (2,3.464) circle(2pt);
\filldraw[fill=black] (-2,3.464) circle(2pt);
\draw [line width=1pt](0,0) -- (2,3.464);
\draw [line width=1pt](0,0) -- (-2,3.464);
\draw [line width=1pt](2,3.464) -- (-2,3.464);
\draw [line width=1pt](0,3.464) -- (1,1.732);
\draw [line width=1pt](0,3.464) -- (-1,1.732);
\draw [line width=1pt](-1,1.732) -- (1,1.732);
\end{tikzpicture}
\end{center}

By Theorem~\ref{p+1 chromatic}, the Tur\'{a}n number of a family of graphs is asymptotically determined by the graph with minimal chromatic number in the family (unless this family of graphs contains a bipartite graph). However, the graphs with larger chromatic number may influence the subtle structure of the extremal graphs. This is one important motivation for studying decomposition family sequences of graphs.




There is a family of graphs such that the extremal graphs of those graphs are determined by its decomposition family provided $n$ is sufficiently large. We say that a graph is {\it edge-critical} if it contains an edge whose deletion reduces the chromatic number of it. In 1968, Simonovits \cite{Simonovits1968} proved the following theorems.

\begin{theorem}\cite{Simonovits1968}\label{Chromatic Critical Edge}
Let $F_1,\ldots,F_\ell$ be given graphs, such that $\chi(F_i)\geq p+1$ $(i=1,\ldots,\ell)$ but there are an $F_{i_o}$ and an edge $e$ in it such that $\chi(F_{j}-\{e\})=p$, where $j\in\{1,\ldots,\ell\}$. Then there exists an $n_0$ such that if $n>n_0$ then $T(n,p)$ is the unique extremal graph for $F_1,\ldots,F_\ell$.
\end{theorem}

Even more, the converse of Theorem~\ref{Chromatic Critical Edge} is also true.
\begin{theorem}\cite{Simonovits1968}\label{Chromatic Critical Edge Inverse}
Let $F_1,\ldots,F_\ell$ be given graphs, if there exists an $n_0$ such that when $n>n_0$ then $T(n,p)$ is the unique extremal graph for $F_1,\ldots,F_\ell$, then $\chi(F_i)\geq p+1$ $(i=1,\ldots,\ell)$ and there are an $F_{j}$ and an edge $e$ in it such that $\chi(F_{j}-\{e\})=p$, where $j\in\{1,\ldots,\ell\}$.
\end{theorem}

Hence, it is interesting to ask the following questions.

\begin{question}\label{Q5}
Find new family of graphs whose decomposition family determines its extremal graphs, or much weaker, find new family of graphs whose first two graph sets of decomposition family sequence determine its extremal graphs.
\end{question}

\begin{question}\label{Q4}
Find extremal graphs such that the forbidden graphs are determined by it in a certain sense. Which graph can be an extremal graph for a finite graph set $\mathcal{L}$?
\end{question}

For Question~\ref{Q5}, as mentioned in Section 1.2, there exist graphs such that the decomposition family of them do not determine the extremal graphs of them.  For Question~\ref{Q4}, we refer the interested readers to \cite{Simonovits1982,Yuan1}.

Although, the definition of decomposition family sequence of a graph is simple, determining the decomposition family sequence of a graph is very complicate. Here we only show that $K_5$ is determined by its decomposition family sequence.

\begin{proposition}
$K_5$ is determined by its decomposition family sequence.
\end{proposition}
\begin{proof} Let $\mathcal{G}_{n,e}$ be the set of graphs with $n$ vertices and $e$ edges. Let $\mathcal{F}=\mathcal{F}_0=\{K_5\}$, by a simple observation, we have
$$
\mathcal{M}_0(\mathcal{F})=\{K_2\},\mbox{ }\mathcal{M}_1(\mathcal{F})=\{K_2\},\mbox{ } \mathcal{M}_2(\mathcal{F})=\{S_3,M_4\},\mbox{ }\mathcal{M}_3(\mathcal{F})=\mathcal{G}_{5,6}$$
and
$$\mathcal{F}_0=\mathcal{G}_{5,10},\mbox{ }\mathcal{F}_1=\mathcal{G}_{5,9},\mbox{ }\mathcal{F}_2=\mathcal{G}_{5,8},\mbox{ } \mathcal{F}_3=\mathcal{G}_{5,6}.
$$
Now let $$
\mathcal{M}_0(\mathcal{F})=\{K_2\},\mbox{ }\mathcal{M}_1(\mathcal{F})=\{K_2\},\mbox{ } \mathcal{M}_2(\mathcal{F})=\{S_3,M_4\},\mbox{ }\mathcal{M}_3(\mathcal{F})=\mathcal{G}_{5,6}.$$
Since $\mathcal{G}_{n,e}$ contains all the graphs with $n$ vertices and $e$ edges and $S_3,M_4$ are the only graphs with two edges (regardless of isolated vertices), we have $\mathcal{F}_3=\mathcal{G}_{5,6}$ and hence  $\mathcal{F}_2=\mathcal{G}_{5,8}$   $\mathcal{F}_1=\mathcal{G}_{5,9}$  and  $\mathcal{F}_0=\mathcal{G}_{5,10}$. The result follows since $\mathcal{G}_{5,10}=\{K_5\}$. \end{proof}

\end{document}